\documentclass[a4paper,final]{amsart}

\usepackage{amsfonts}
\usepackage{amssymb}
\usepackage[latin2]{inputenc}
\usepackage{enumerate}
\usepackage{amsmath}
\usepackage[notcite,notref]{showkeys}
\usepackage{amssymb}
\newtheorem{theorem}{Theorem}[section]
\newtheorem{lemma}[theorem]{Lemma}

\newtheorem{question}[theorem]{Question}
\newtheorem{cor}[theorem]{Corollary}

\theoremstyle{definition}                   
\newtheorem{remark}[theorem]{Remark}
\newtheorem{example}[theorem]{Example}
\newtheorem{defi}[theorem]{Definition}
\newtheorem{notation}[theorem]{Notation}

\newcommand{\Z}{\mathbb{Z}}
\newcommand{\Q}{\mathbb{Q}}
\newcommand{\R}{\mathbb{R}}

\newcommand{\eps}{\varepsilon} 
\newcommand{\de}{\delta}
 
\newcommand{\cah}{\mathcal{H}}
\newcommand{\cak}{\mathcal{K}}
\newcommand{\K}{\mathcal{C}}
\newcommand{\A}{\mathcal{A}}
\newcommand{\KK}{\widetilde{\mathcal{C}}}
\newcommand{\inte}{\text{int}\,}

\newcommand{\Var}{\mathop{Var}}
\def\beq{\begin{equation}}
\def\eeq{\end{equation}}
\newcommand{\norm}[1]{\left\lVert #1 \right\rVert} 
\newcommand{\abs}[1]{\left\lvert #1 \right\rvert} 
\newcommand{\sk}[1]{\langle #1 \rangle} 
\newcommand{\intnorm}[2]{\sup_{t\in[x - #2, x+ #2]}\left|\int_x^t #1\right|}
\newcommand{\clK}{{\overline{K}}}

\newcommand{\bt}{\begin{theorem}}
\newcommand{\bl}{\begin{lemma}}
\newcommand{\br}{\begin{remark}}
\newcommand{\bd}{\begin{defi}}

\newcommand{\et}{\end{theorem}}
\newcommand{\el}{\end{lemma}}
\newcommand{\er}{\end{remark}}
\newcommand{\ed}{\end{defi}}
\newcommand{\bp}{\begin{proof}}
\newcommand{\ep}{\end{proof}}

\newcommand{\lab}[1]{\label{#1}}

\DeclareMathOperator{\diam}{diam}

\DeclareMathOperator{\supp}{supp}

\newcommand{\leb}{\lambda}
\newcommand{\su}{\subset}

\newcommand{\RR}{\mathbb{R}}

\newcommand{\iD}{\mathcal{D}}

\newcommand{\iH}{\mathcal{H}}

\newcommand{\iJ}{\mathcal{J}}

\newcommand{\sm}{\setminus}
\newcommand{\la}{\lambda}
\newcommand{\al}{\alpha}
\renewcommand{\phi}{\varphi}

\title[Reconstructing geometric objects]{Reconstructing geometric objects from
  the measures of their intersections with test sets}

\author{M\'arton Elekes}
\address{Alfr\'ed R\'enyi Institute of Mathematics\\
PO Box 127, 1364 Budapest, Hungary\\
and
Institute of Mathematics\\
E\"otv\"os Lor\'and University\\
P\'azm\'any P\'eter s.~1/c, 1117
Budapest, Hungary}
\email{elekes.marton@renyi.mta.hu}
\urladdr{http://www.renyi.hu/~emarci}
\thanks{We gratefully acknowledge the support of the
Hungarian Scientific Foundation grants no.~72655, 61600, 83726 and J\'anos
Bolyai Fellowship. The third author was supported by the EPSRC grant EP/G050678/1.
}

\author{Tam\'as Keleti}
\address{Institute of Mathematics\\
E\"otv\"os Lor\'and University\\
P\'azm\'any P\'eter s.~1/c, 1117
Budapest, Hungary}
\email{tamas.keleti@gmail.com}
\urladdr{http://www.cs.elte.hu/analysis/keleti}

\author{Andr\'as M\'ath\'e}
\address{Mathematics Institute, University of Warwick \\
Coventry, CV4~7AL, United Kingdom}
\email{A.Mathe@warwick.ac.uk}
\urladdr{http://www.warwick.ac.uk/~masibe}

\begin{document}
\begin{abstract}
Let us say that an element of a given family $\A$ of subsets of $\R^d$ can be
reconstructed using $n$ test sets if there exist $T_1,\ldots,T_n \su \R^d$
such that whenever $A,B\in \A$ and the Lebesgue measures of $A \cap T_i$ and
$B \cap T_i$ agree for each $i=1,\ldots,n$ then $A=B$.  Our goal will be to
find the least such $n$.

We prove that if $\A$ consists of the translates of a fixed reasonably
nice subset of $\R^d$ then this minimum is $n=d$.  In order to obtain
this result, on the one hand we reconstruct a translate of a fixed
absolutely continuous function of one variable using $1$ test set. On
the other hand, we prove that under rather mild conditions the Radon
transform of the characteristic function of $K$ (that is, the measure
function of the sections of $K$), $(R_\theta \chi_K) (r) = \la^{d-1} (K
\cap \{x \in \RR^d : \langle x,\theta \rangle = r\})$ is absolutely
continuous
for almost every direction $\theta$. These
proofs are based on techniques of harmonic analysis.

We also show that if $\A$ consists of the magnified copies $rE+t$ $(r\ge 1,
t\in\R^d)$ of a fixed reasonably nice set $E\su \R^d$, where $d\ge 2$, then
$d+1$ test sets reconstruct an element of $\A$, and this is optimal.  
This fails in $\R$: we prove
that a closed interval, 
and even a closed interval of length at least $1$ cannot be
reconstructed using $2$ test sets.

Finally, using randomly constructed test sets, we prove that an element of a
reasonably nice $k$-dimensional family of geometric objects can be 
reconstructed using $2k+1$ test sets.  An example from algebraic topology shows
that $2k+1$ is sharp in general.
\end{abstract}



\maketitle 

\section{Introduction}

There is a vast literature devoted to various kinds of geometric
reconstruction problems. Part of the reasons why these are so popular is their
connection with geometric tomography.

The set of reconstruction problems we will study is the following. 
Given a family of subsets of $\R^d$ we would like to find ``test sets''
such that whenever someone picks a set from the family and hands us the 
Lebesgue measure of the chunk of this set in the test sets, then we can
tell which the chosen set is. 
In other words, the measures of the intersection of the set with our
test sets uniquely determine the member of the family.
Our aim is to use as few test sets as possible. 
If it is enough to use $n$ test sets then we say
that \emph{an element of the given family can be reconstructed using $n$ test
sets}. The formal definition is the following. We denote the Lebesgue
measure on $\R^d$ by $\la^d$.

\begin{defi}
Let $\A$ be a family of Lebesgue 
measurable subsets of $\R^d$ of finite measure. We say that 
\emph{an element of $\A$ can be reconstructed using $n$ test sets}
if there exist measurable sets $T_1,\ldots,T_n$ such that 
whenever $A,B\in \A$ and $\la^d(A \cap T_i)=\la^d(B \cap T_i)$ for
every $i=1,\ldots,n$ then $A=B$.
\end{defi}  

The first question of this form we are aware of is the following folklore
problem, which asks, using the above terminology, 
whether an axis parallel unit subsquare of $[0,10]\times[0,10]$
can be reconstructed using two test sets. We leave this question to the reader
as an exercise. There are numerous natural modifications of the problem: 
Can a unit segment of $[0,10]$ (or of $\R$) be reconstructed using
$1$ test set? Can a unit disc be reconstructed using $2$ test sets? What
happens in higher dimensions? And so on.

In each of the above problems the given family $\A$ is the set
of translates of a fixed set. Since in $\R^d$ this means $d$ parameters,
we might hope that we can reconstruct a translate of a fixed set using $d$
test sets. One of our main goals is to show that this is indeed true, 
at least under some mild assumptions on the set. For $d\ge 3$ we prove
(Corollary~\ref{c:posmeasure}) this for any bounded measurable set of positive
measure, in the plane (Theorem~\ref{t:rectbound}) for any 
bounded measurable set of positive measure with rectifiable boundary of 
finite length, and in the real line for any finite union of intervals.

The first idea behind all of the above results for $d\ge 2$ is the
following.  Suppose we want to reconstruct a translate of $E\su\R^d$.
Let $T$ be a test set of the form $T=S\times\R^{d-1}$, where $S\su\R$.
Then clearly $\la^d((E+x)\cap T)$ depends only on the first coordinate
of $x$: in fact, one can easily check that \beq\label{e:integral}
\la^d((E+x)\cap T)=\int_S (R_{(1, 0, \ldots, 0)} \chi_E) (t-x_1) dt, \eeq
where $x_1$ denotes the first coordinate of $x$ and $(R_{(1, 0, \ldots,
0)} \chi_E) (t) = \la^{d-1} (K \cap \{x \in \RR^d : x_1 = t\})$ is the
Radon transform of $\chi_E$ in direction $(1, 0, \ldots, 0)$. (In general, 
the Radon transform of a function $f$ is $(R_\theta f) (t) = 
\int_{\{x \in \RR^d : \langle x,\theta \rangle = t\}} f \,d\la^{d-1}$, but will we only apply 
this for characteristic functions.) Thus if
$\int_S (R_{(1, 0, \ldots, 0)} \chi_E) (t-x_1) dt$ uniquely determines
$x_1$, in other words if we can reconstruct a translate of the
$\R\to[0,\infty]$ function $R_{(1, 0, \ldots, 0)} \chi_E$ using the test
set $S\su\R$ then $\la^d((E+x)\cap T)$ determines $x_1$.  Therefore, if
we can do this in $d$ linearly independent directions then we are done.

In order to carry out the above program, first we show (Theorem~\ref{t1}) that 
a translate of a fixed
non-negative not identically zero compactly supported absolutely continuous 
function can be reconstructed using one test set in the above sense. 
Then we get the above results by finding many directions
in which the Radon transform of the characteristic function 
is absolutely continuous.
For concrete sets (for example if 
we want to reconstruct a translate of the unit ball) this is immediate, 
for more general sets we use Fourier transforms.

We also consider the problem of reconstruction of a magnified copy of a fixed
set $E\su\R^d$ ($d\ge 2$), where by a magnified copy of $E$ 
we mean a set of the form $rE+x$, where $r\ge 1$ and $x\in\R^d$.
Since we have $d+1$ parameters, one can hope to
reconstruct using $d+1$ test sets.
Using $\R^d$ as a test set we can reconstruct $r$ since $\la^d(rE+x)$
depends only on $r$. 
The reconstruction of $x$ 
is done similarly as in the above case of translations but now we
need to consider not only translations of the Radon transforms
but also the translations of their rescaled copies, since 
instead of \eqref{e:integral} we 
need here the more general $\la^d((rE+x)\cap T)=r^{d-1}\int_S (R_{(1, 0, \ldots, 0)} \chi_E) (\frac{t-x_1}{r}) dt$.
Therefore we need to choose the set $S\su\R$ so that for \emph{every} $r\ge 1$
the integral $\int_S (R_{(1, 0, \ldots, 0)} \chi_E) (\frac{t-x_1}{r}) dt$ determines $x_1$. 
So this is a harder task than in the case of translations (where we needed this only
for $r=1$) and we can only prove a positive result (Theorem~\ref{t2})
under some additional assumption on the derivative of the Radon transform: we also 
require that $(R_\theta \chi_E)'$ can be approximated well in $L^1$ norm  by a $g\in C^1$ function with small
$\norm{g'}_1$. For functions obtained from concrete nice sets of $\R^d$ 
($d\ge 2$)  
(for example, if we want to reconstruct a
ball of radius at least $1$) this condition can be checked. 
For the
more general case we again have to find many directions in which this 
condition is satisfied.
We have positive general result if we assume that $d\ge 4$, or 
that $d\ge 2$ and $E$ is convex.
In the first case we use Fourier transforms here as well, while in the
second case we take advantage of the Brunn--Minkowski inequality among others.
This way we get (Corollaries~\ref{c:genmagn}~and~\ref{c:convexmagn}) that 
if $E\su\R^d$ is a fixed bounded measurable set
of positive measure and $d\ge 4$, or if $E\su\R^d$ is a convex 
set with nonempty interior and $d\ge 2$ then 
a set of the form $rE+x$, where
$r\ge 1$ and $x\in\R^d$ can be reconstructed using $d+1$ sets.

In all of the above mentioned results the reconstruction is impossible using
fewer test sets, since that would mean a continuous injective map
from the parameter space into a smaller dimensional Euclidean space:
if we attempt to reconstruct an element of $\{A_{\al}:\al\in\Lambda\}$,
where the parametrization is chosen so that $\al\mapsto\la^d(A_\al\cap T)$
is continuous for any measurable set $T\su\R^d$ then reconstruction 
using $T_1,\ldots,T_n$ would yield 
that $\al\mapsto(\la^d(A_\al\cap T_1),\ldots,\la^d(A_\al\cap T_n))$
is a continuous injective $\Lambda\to\R^n$ map, which is impossible
if $\Lambda$ contains an open subset of $\R^{n+1}$, or more
generally an $(n+1)$-dimensional manifold.

The above argument also shows that sometimes we need more functions than the
number of parameters: if the above $\Lambda$ cannot be embedded continuously
into $\R^n$ then one cannot reconstruct an element of 
$\{A_{\al}:\al\in\Lambda\}$ using $n$ test sets.

\begin{example}\label{e:2k}
Let $\Lambda$ be the $k$-skeleton of a $2k+2$ simplex (that is, $\Lambda$
is the union of the $k$-dimensional faces of a simplex in $\R^{2k+2}$).
By the van~Kampen--Flores Theorem (see e.g.~in \cite{Matousek}) 
$\Lambda$ cannot be embedded continuously into $\R^{2k}$, 
so the above argument shows
that a unit ball in $\R^{2k+2}$ centered at a point of $\Lambda$ cannot 
be reconstructed using $2k$ sets, although the parameter space is
$k$-dimensional, which means that we only have $k$ parameters.
\end{example}

In Section~\ref{freedom} we prove (Theorem~\ref{special case})
that reasonably nice geometric objects 
parametrized reasonably nicely by $k$ parameters
can be reconstructed using $2k+1$ test sets,
which is sharp according to the above example.
The test sets are given using a random construction.
As applications we get 
for example that an $n$-gon in the plane can be reconstructed 
using $4n+1$ test sets, an ellipsoid in $\R^3$ can be reconstructed 
using $19$ test sets, and a ball in $\R^d$ can be reconstructed using $2d+3$
test sets, in particular an interval in $\R$ 
can be reconstructed using $5$ test sets. 
(Here and in the sequel $n$-gons, ellipsoids, balls and intervals are
understood to be closed.)

One might be tempted to say that Example~\ref{e:2k} is quite artificial and in natural
situations the number of parameters should suffice, and probably an
interval
can be reconstructed using $2$ test sets. But this is false, we prove 
(Theorem~\ref{t:unitint}) that an
interval cannot be reconstructed using 
$2$ test sets, not even if we consider only intervals of length more than $1$.
Like above, the obstacle is again of topological nature, although it is 
more complicated since the parameter space can be embedded into $\R^2$.

\begin{remark}
It is natural to ask what happens if we try to reconstruct a set using test 
\textit{functions} by considering the integrals of the functions over the set, or
even test measures by considering the measures of the set.

First we describe why a nice geometric object can be reconstructed using the 
following single test measure.
Let
$$
\mu = \sum_{i=1}^\infty \frac{\delta_{q_i}}{3^i}, 
$$
where $\{q_1, q_2, \ldots\} = \Q^d$ and $\delta_x$ denotes the Dirac measure
at $x$.
Then it is easy to see that $\mu(A) = \mu(B) \iff A \cap \Q^d = B \cap \Q^d$.
Suppose that 
the symmetric difference of any two distinct sets of $\A$ contains
a ball, which is always the case for families of geometric objects.
Then $\mu$ reconstructs a member of $\A$, that is, 
whenever $A,B\in\A$ are distinct sets then $\mu(A)\neq\mu(B)$.

Similarly to the case of test sets, it is not possible to reconstruct
a member of $\A$
with fewer bounded test functions than the dimension of the parameter space
of $\A$.
However, as opposed to the case of test sets, 
it is almost obvious to reconstruct a translate of a fixed  
bounded measurable set $E\su\R^d$
using $d$ bounded test functions: 
the functions 
$\arctan x_1,\ldots,\arctan x_d$ reconstruct a translate of $E$.   

It is also very easy to reconstruct a magnified copy $rE+t$ of a fixed 
bounded measurable set 
$E\su\R^d$
using $d+1$ bounded test functions: the constant $1$ function determines 
$r$, and then 
$\arctan x_1,\ldots,\arctan x_d$ determine $t$. As it was mentioned earlier,
finding $d+1$ test sets that reconstruct a magnified copy of a fixed set 
is much harder in higher dimensions, and it is even impossible in $\R$: an interval
cannot be reconstructed using two test sets.

Therefore the reconstruction problem for intervals shows nicely the difference 
between test measures, test functions and test sets: an interval of $\R$
can be reconstructed using $1$ test measure, it can be reconstructed using
$2$ test functions (but $1$ does not suffice) and it cannot be reconstructed using less than $3$ test 
sets.  
\end{remark}

 
%

\section{Reconstruction of an interval}\label{s:int}

By interval we always mean a bounded nondegenerate closed interval.
The following simple lemma is the key tool to reconstruct a translate
of a fixed interval or a finite union of intervals.

\begin{lemma}\label{l:semigroup}
Suppose that $G\su (0,\infty)$, $h\in G$ and $G+h\su G$.
Let $A\su \R$ be a measurable set that has positive Lebesgue measure in
every nonempty interval. Suppose that $A\cap (A+G)=\emptyset$.
Then an interval of length $h$ can be reconstructed using the test set $T=A\cup
(A+G)$;
in fact, $x\mapsto\la([x,x+h]\cap T)$ is strictly increasing.
\end{lemma}

\begin{proof}
Clearly it is enough to prove that 
$\la([u,u+h]\cap T) < \la([v,v+h]\cap T)$ for any $u<v<u+h$.
So let  $u<v<u+h$. Then
$$
\la([v,v+h]\cap T) - \la([u,u+h]\cap T) = 
\la([u+h,v+h]\cap T) - \la([u,v]\cap T) 
$$
$$
=\la([u+h,v+h]\cap A) + \la([u+h,v+h]\cap (A+G)) 
- \la([u,v]\cap (A \cup (A+G))).
$$
The first term is positive since $A$ has positive measure in
every nonempty interval. By the translation invariance of $\la$,
the second term can be written as
$\la([u,v]\cap (A+G-h))$. 
Hence it is enough to prove that $A \cup (A+G) \su A+G-h$.
We have $A\su A+G-h$ since $h\in G$, 
and we have $A+G\su A+G-h$ since $G+h\su G$.
\end{proof}

\begin{theorem}\label{t:unitint}
A translate of a fixed interval can be reconstructed using $1$ test set; 
that is, for any $h$ there exists a set $T\su\R$ such that 
$\la([x,x+h]\cap T)\neq \la([y,y+h]\cap T)$ if $x\neq y$.
\end{theorem}

\begin{proof}
We can clearly suppose that $h=1$.
It is not hard to see that one can 
choose countably many pairwise disjoint measurable subsets of $[0,1]$ such that 
all of them have positive measure in every nonempty subinterval of $[0,1]$.
Using $\Z$ as the index set we denote them by $A_k$ ($k\in \Z$).
Then  Lemma~\ref{l:semigroup} applied to $A=\cup_{k\in\Z} (A_k +k)$,
$G=\{1,2,\ldots\}$, $h=1$ completes the proof. 
\end{proof}

To reconstruct a translate of a fixed finite union of intervals,
Lemma~\ref{l:semigroup} has to be applied for a more complicated $G$, for
which it is a bit harder to construct a suitable $A$.
This is done in the following lemma.
 
We call a set $E\su \R$ \emph{locally finite} if it has finitely many elements
in every bounded interval.

\begin{lemma}\label{l:discrete}
For any locally finite set $G\su (0,\infty)$
there exists a measurable set $A\su \R$ such that 
$A$ has positive Lebesgue measure in
every nonempty interval and $A\cap (A+G)=\emptyset$.
\end{lemma}
\begin{proof}
Let $I_1,I_2,\ldots$ be an enumeration of the intervals with rational
endpoints of length less than the minimal element of $G$. 

By induction we define nowhere dense closed sets $A_1,A_2,\ldots$
with positive measure such that for every $n$, $A_n\su I_n$ and 
\beq\label{e:union}
\left(\bigcup_{j=1}^n A_j\right) \cap \left(\bigcup_{j=1}^n A_j + G \right) = 
\emptyset.
\eeq
This will complete the proof since then we can choose $A=\cup_{j=1}^\infty A_j$.

We can take $A_1$ as an arbitrary nowhere dense closed subset of $I_1$
with positive measure since then \eqref{e:union} is guaranteed by 
$\diam(I_1)< \min G$.

Suppose that we already chose $A_1,\ldots,A_{n-1}$ with all the requirements
up to $n-1$. For any $A_n\su I_n$ we have $A_n \cap (A_n+G)=\emptyset$.
To complete the proof we need to choose a nowhere dense closed set $A_n\su I_n$ of positive measure
disjoint to $(\cup_{j=1}^{n-1} A_j) + G$ and $ (\cup_{j=1}^{n-1} A_j) - G$.
Since $G$ is locally finite, we only need to avoid the union of finitely many translates
of the nowhere dense closed set $\cup_{j=1}^{n-1} A_j$. As this union is closed and nowhere dense, it is not of full measure in $I_n$.
\end{proof}

\begin{theorem}\label{t:finiteint}
Let $E$ be a finite union of intervals in $\R$.
Then a translate of $E$ can be 
reconstructed using $1$ set; that is,
there exists a measurable set $T$ such that 
$\la((E+t)\cap T)\neq  \la((E+t')\cap T)$ if
$t\neq t'$.
\end{theorem}
\begin{proof}
Let $E=\cup_{j=1}^n I_j$, where $I_j$ is an interval of length $a_j$ 
and the intervals are pairwise disjoint.
Let $G$ be the additive semigroup generated by $a_1,\ldots,a_n$;
that is,
$G=\{\sum_{i=1}^n k_i a_i \ : \ k_1,\ldots,k_n\in\{0,1,2,\ldots\}\} \sm \{0\}$.
Then $G\su (0,\infty)$ is a locally finite set and it contains every $a_i$.
Let $A$ be the set obtained by Lemma~\ref{l:discrete} from $G$
and let $T=A\cup(A+G)$.
Then by Lemma~\ref{l:semigroup} each function 
$x\mapsto \la((I_j+x)\cap T)$ is strictly increasing, so
their sum $x\mapsto \la((E+x)\cap T)$ is also strictly increasing,
which completes the proof.
\end{proof}

An interval has two parameters, so one cannot reconstruct an interval using
$1$ test set, but one might expect that $2$ test sets should be enough. 
We show that this is false. 
The following lemma concerns the topological obstacle of the
reconstruction using two sets.
The lemma is surely well known for topologists, but for completeness
we present a short proof. 

\begin{lemma}\label{findreference}
Let $U\subset\R^2$ be a path connected open set and let $f:U\to\R^2$ be continuous and injective. Suppose that $f$ is differentiable at two points $a$ and $b$ such that the determinant of the Jacobi matrix at $a$, $\det f'(a)>0$. Then $\det f'(b)\ge 0$.
\end{lemma}
\begin{proof}
%
Suppose that $\det f'(b)<0$.
Let $C:[0,2\pi]\to \R^2$ be the curve $C(t)=e^{it}$. If $r$ is small enough, the winding number of the curve $f(a+rC)$ around $f(a)$ is $1$, while the winding number of $f(b+rC)$ around $f(b)$ is $-1$. However, $U$ is path-connected and the winding number of $f(x+rC)$ is continuous in $x$, which yields a contradiction.
\end{proof}

\begin{theorem}\label{t:interval}
An interval in $\R$ cannot be reconstructed using two measurable sets.
Moreover, even an interval of length bigger than $1$ cannot be reconstructed
using two sets;
that is, for any pair of measurable sets $A, B \su \RR$ there exist two
distinct intervals $I$ and $I'$  of length bigger than $1$ such that
$\leb( I \cap A ) =\leb( I' \cap A )$ and $\leb( I \cap B ) =\leb( I' \cap B )$.
\end{theorem}
\begin{proof}
Suppose to the contrary that $A$ and $B$ reconstruct
an interval of length bigger than $1$. 
Let $U=\{(x,y)\in\R^2: y-x>1\}$. Let $f:U\to [0,\infty)^2$ be defined by $$f((x,y))=(\lambda(A\cap [x,y]), \lambda(B\cap[x,y])).$$ The map $f$ is Lipschitz, and since $A$ and $B$ reconstruct, it is also injective.

Let $d_H(x)=\lim_{r\to 0+}\la(H\cap [x-r,x+r])/2r$ 
denote the density of a set $H$ at a point $x$ if the limit exists.  
Suppose that $y-x>1$ and $d_A(x), d_B(x), d_A(y), d_B(y)$ all exist. Using the $o$ notation,
\begin{align*}
f(x+t_x, y+t_y)= (&\lambda(A\cap[x,y])- d_A(x)t_x + d_A(y)t_y + o(t_x) + o(t_y), \\
 &\lambda(B\cap[x,y])- d_B(x)t_x + d_B(y)t_y + o(t_x) + o(t_y) ).
\end{align*}
This implies that $f$ is differentiable at $(x,y)$ and its derivative (Jacobian) is
$$ \left(
\begin{array}{cc}
-d_A(x) & d_A(y)  \\
-d_B(x) & d_B(y) \\
\end{array} \right).
$$

Let $I_1$ and $I_2$ be two non-empty intervals such that their distance is bigger
than $1$ and $I_1$ is on the left-hand side of $I_2$. 
Then none of 
$A, B, A^c, B^c$ and $A\triangle B$ 
can have zero measure intersection with both of $I_1$ and $I_2$,
since otherwise $f$ maps $I_1 \times I_2 \su U$ injectively and continuously 
into a (vertical, horizontal or diagonal) line, which is impossible.
This implies that all of $A, B, A^c, B^c$ and $A\triangle B$ must have positive measure in any interval of length bigger than $1$.

In particular, $\lambda(A \triangle B)>0$ and both $A$ and $B$ have 
positive measure in every halfline.

Since $\lambda(A \triangle B)>0$, we have $\lambda(A\setminus B)>0$ or
$\lambda(B\setminus A)>0$. We may suppose that the first one holds.
Recall that Lebesgue's density theorem states that 
the density of a measurable set is $1$
at almost all of its points and $0$ at almost all of the points of its 
complement.
Since $\lambda(A\setminus B)>0$, this implies that there exists a point $z$
for which $d_{A\sm B}(z)=1$. 
Then $d_A(z)=1$ and $d_B(z)=0$.
Since $B\cap (-\infty,z-1)$ and $B\cap (z+1,\infty)$ have
positive measure, we can pick $u<z-1$ and $v>z+1$
such that $d_B(u)=d_B(v)=1$ and both of $d_A(u)$ and $d_A(v)$ exist. 

Then
$$ f'(z,u) = \left(
\begin{array}{cc}
-1 & d_A(u)  \\
0 & 1 \\
\end{array} \right)
\quad \text{and} \quad
 f'(v,z) = \left(
\begin{array}{cc}
-d_A(v) & 1  \\
-1 & 0 \\
\end{array} \right),
$$
thus $\det f'(z,u)=-1$, $\det f'(v,z)=1$. 
This contradicts Lemma~\ref{findreference}.
\end{proof}

In Corollary~\ref{c:concrete} we will see that $5$ test sets are enough.
We do not know whether $3$ or $4$ are enough or not.

\section{Reconstruction of a translate of a fixed function}

As it is explained in the Introduction, for getting positive results
about the reconstruction of a translate of
a fixed set in $\R^d$ ($d\ge 2$), it will be useful to get 
results about the reconstruction of a translate of a given $\R\to\R$ function
using $1$ test set.

To reconstruct a translate of a fixed function, 
the following definition will be crucial.

\begin{defi}
Let $f:\R\to\R$ be an $L^1$ function and $\eps > 0$. Define
\[
K(\eps,f) =
 \inf \left\{ \Var(g) : g \textrm{ is compactly supported,} \, \norm{f-g}_1 <\eps \right\},
\]
where $\Var(g)$ denotes the total variation of $g$.
\end{defi}

Clearly, $K(\eps, f)$ is monotone in $\eps$. Also, $K(\eps,f)$ is always
finite as the piecewise constant functions of compact support are dense in $L^1$.

The following lemma shows that we can replace functions of bounded variation
by $C^1$ functions.

\bl
\lab{l:C^1}
Let $f:\R\to\R$ be an $L^1$ function with $\supp(f) \su [-1,1]$ and $\eps > 0$. Then
\[
K(\eps,f)= \inf \left\{ \Var(g) : g \in C^1,
  \supp(g) \su [-1,1], \,\norm{f-g}_1 <\eps \right\}. 
\] 
\el

Note that if $g \in C^1$ then $\Var(g) = \norm{g'}_1$.

\bp
It suffices to prove that if $g$ is of bounded variation with $\supp(g) \su [-1,1]$ and $\eps > 0$ then
there exists a $g_1 \in C^1$ with $\supp(g_1) \su [-1,1], \norm{g - g_1}_1 < \eps$ and $\Var(g_1) \le \Var(g)$. 

Let us first assume instead that $g$ is constant on $(- \infty, -1)$ and $(1, \infty)$, and it is also monotone. It is not hard to find a piecewise constant monotone function $g_0$ such that $\norm{g - g_0}_1 < \eps$. Then clearly $\Var(g_0) = \Var(g)$. Finally, we can easily approximate $g_0$ by a monotone $g_1 \in C^1$ such that $\norm{g_0 - g_1}_1 < \eps$ and $\Var(g_1) = \Var(g)$. 

Let now $g$ be a general function of bounded variation with $\supp(g) \su [-1,1]$. It is well-known that it can be decomposed as $g = g_+ - g_-$, where $g_+$ and $g_-$ are non-decreasing and $\Var(g) = |g_+(1) - g_+(-1)| + |g_-(1) - g_-(-1)|$ (indeed, let $g_+(x)$ be the positive variation of $g$ on $[-1, x]$). Applying the above approximation gives the result.
\ep

Recall that $f * g$ stands for the convolution of the two functions, and also that a function $f$ is locally absolutely continuous iff there exists a
locally $L^1$
function  $f^*$ such that $f(y) - f(x)= \int_x^y f^*(t) \,dt$ for every $x, y \in \RR$. Moreover, in that case $f^* = f'$ almost everywhere.
The following lemma is rather well-known, but we were unable to find a suitable reference so we include a proof.

\bl
\lab{l:conv}
Let $f : \RR \to \RR$ be locally absolutely continuous, $g : \RR \to \RR$ be locally $L^1$ and assume that one of them is compactly supported. Then $f * g$ is also locally absolutely continuous and $(f * g)' = f' * g$ almost everywhere. Moreover, if $g$ is locally $L^\infty$ then $f*g$ is $C^1$ and $(f * g)' = f' * g$ everywhere.
\el

\bp 
Since we are only interested in the local behaviour of $f*g$, and one of them is compactly supported, we may actually assume (using the formula defining $f*g$) that both of them are compactly supported. This justifies the use of Fubini's Theorem in the following computation. 
\[
(f*g)(y) - (f*g)(x) = \int_\RR [f(y-u) - f(x-u)] g(u) \,du = \int_\RR \int_{x-u}^{y-u} f'(t) \,dt\  g(u) \,du =
\]
\[
\int_\RR \int_x^y f'(t-u) \,dt \ g(u) \,du = \int_x^y  \int_\RR f'(t-u)  g(u) \,du \,dt =  \int_x^y  (f' * g)(t) \,dt,
\]
hence we are done with the proof of the first statement. If $g$ is also in $L^\infty$, then $f'*g$ is the convolution of an $L^1$ and an $L^{\infty}$ function, so it is continuous. The above equation shows that $f*g$ is the integral 
of this continuous function,
which yields the remaining statements.
\ep

\begin{lemma}\label{l:77}
Let $f : \RR \to \RR$ be a non-negative 
absolutely continuous function with $\supp(f)\su[-1,1]$ and $\int_{\R}f=1$.
Let $a>0$ and $f_a(x)=f(x/a)/a$. 
Let $\Phi:\R\to (0,1)$ be a $C^1$ function with $\Phi'>0$, 
and let $\Psi:\R\to[0,1]$ be a measurable function.
Then for any  $\eps>0$ and $x\in\R$ we have
$$
(f_a * \Psi)'(x) \ge  \min_{[x-a,x+a]} \Phi' -\frac{2\eps}{a} - 
\frac{K(\eps, f')}{a^2} \intnorm{\Psi-\Phi}{a}.
$$
\end{lemma}

\bp
We denote by $f_a'$ the $L^1$ function for which $f_a(x)=\int_{-\infty}^x f_a'(t) \,dt$. Clearly $f_a'=f'(x/a)/a^2$.
Since $\supp(f)\su[-1,1]$, the function $f_a$ is supported in $[-a,a]$.

Fix $\delta>0$. By Lemma~\ref{l:C^1}, we can 
choose a $C^1$ function $g_0$ supported in $[-1,1]$ such that 
$\norm{f'-g_0}_1<\eps$ and $\norm{g_0'}_1\le K(\eps, f') +\delta$. 
Let $g(x)=g_0(x/a)/a^2$ (thus $g'(x)=g_0'(x/a)/a^3$).
Then we have
\begin{equation}
\label{two}
\norm{f_a'-g}_1<\eps/a \quad \textrm{and} \quad
\norm{g'}_1 \le \frac{K(\eps, f')+\delta}{a^2}.
\end{equation}

Using Lemma \ref{l:conv} several times, we obtain
\begin{align*}
(f_a * \Psi)'(x) & = (f_a * \Phi)'(x) + (f_a * (\Psi-\Phi))'(x) = \\
& = (f_a * \Phi')(x) + (f_a' * (\Psi-\Phi))(x) = \\
&=  (f_a * \Phi')(x) + ((f_a'-g) * (\Psi-\Phi))(x) +  (g * (\Psi-\Phi))(x)  = \\
&=  (f_a * \Phi')(x) + ((f_a'-g) * (\Psi-\Phi))(x) + \left (g' * \int (\Psi-\Phi)\right)(x), 
\end{align*}
where $\left(\int (\Psi-\Phi)\right)(t)=\int_{t_0}^t (\Psi(s)-\Phi(s)) ds$
for an arbitrary fixed $t_0$. 
Using $\int_{\R}f_a=1$,  
$|\Psi-\Phi|\le 2$ and that $f_a$ and $g'$ are supported 
in $[-a,a]$, and then (\ref{two}),
this implies that
\begin{align*}
(f_a * \Psi)'(x) & \ge  \min_{[x-a,x+a]} \Phi' - 2 \norm{f_a'-g}_1 - \norm{g'}_1 \intnorm{\Psi-\Phi}{a} \\
& \ge  \min_{[x-a,x+a]} \Phi' -\frac{2\eps}{a} - 
\frac{K(\eps, f')+\delta}{a^2} \intnorm{\Psi-\Phi}{a}.
\end{align*}
Letting $\delta\to 0$ we get the claimed inequality.
\ep

In this section our goal is to reconstruct a translate of a given 
non-negative not identically zero compactly supported absolutely 
continuous function $f$ on $\R$ by a measurable set $T$ by choosing
$T$ so that $\int_T f(x-b)dx$ is strictly increasing in $b$. 
Note that $\int_{\R} \Phi(x)f(x-b)dx$ is strictly increasing for any
strictly increasing $\Phi(x)$, and by denoting the characteristic 
function of $T$ by $\chi_T$, 
we have $\int_T f(x-b)dx=\int_{\R} \chi_T f(x-b)dx$. 
Therefore our task is to approximate a given $\Phi$ by a characteristic
function, so that their integrals are close. This will be done in the
following two lemmas.    

\begin{lemma}\label{l:Marci}
Let $\Phi:[0,1]\to (0,1)$ be a $C^1$ function with $\Phi'>0$, and let $\delta
> 0$.
Then we can choose $T\su[0,1]$ as a finite union of intervals so that 
\beq\label{delta}
\abs{\int_{a}^b (\chi_T - \Phi)} \le \de \text{ for any } a,b\in[0,1]
\eeq
and
\beq\label{zero}
\int_{0}^{1} (\chi_T - \Phi)=0.
\eeq
\end{lemma}

\bp
Choose a positive integer $n$ so that $4/n<\de$.
Let $T_0=\emptyset, T_1=[0,1/n]$ and by induction construct $T_m\su[0,m/n]$ 
for $m=1,\ldots,n$  
so that $T_m=T_{m-1}$ or $T_m=T_{m-1}\cup[(m-1)/n,m/n]$ and
$0 \le \int_0^{m/n} (\chi_{T_m} - \Phi) \le 1/n$. 
Then letting $h=\int_0^1 (\chi_{T_n} - \Phi)$ we have $0\le h \le 1/n$.
Since $T_n\supset T_1=[0,1/n]$, by letting $T=T_n\sm [0,h]$ we 
have (\ref{zero}) and
$-1/n \le \int_0^{m/n} (\chi_{T_m} - \Phi) \le 1/n$ for any $m=1,\ldots,n-1$.
Then $-2/n \le \int_0^{b} (\chi_{T_m} - \Phi) \le 2/n$ for any $b\in[0,1]$,
which implies (\ref{delta}) since $4/n<\de$. 
\ep

Recall that $\lfloor x \rfloor$ denotes the integer part of $x$.

\begin{lemma}\label{l:88}
Let $\Phi:\R\to (0,1)$ be a $C^1$ function with $\Phi'>0$ and 
$\de:\{0,1,2,\ldots\} \to (0,1)$. 
Then  we can choose $T$ as a locally finite union of intervals so that 
\beq\label{defA}
\abs{\int_{a}^b (\chi_T - \Phi)} \le \de(\lfloor |a| \rfloor) + \de(\lfloor
|b| \rfloor)
\eeq
for any $a,b\in\R$.
\end{lemma}

\bp
Apply Lemma~\ref{l:Marci} on $[k,k+1]$ and on $[-k-1,-k]$
(instead of $[0,1]$) and $\de=\de(k)$ for each $k=0,1,\ldots$ and 
let $T$ be the union of the sets we obtain. 
\ep

Now we can prove the main result of the section.

\begin{theorem}\label{t1}
Let $f : \RR \to \RR$ be a non-negative not identically zero compactly supported absolutely continuous function. Then a translate of $f$ can be reconstructed using one test set; that is,
there exists a measurable set $T$ such that if $b \neq b'$ then $\int_T f(x-b) \,dx \neq \int_T f(x-b') \,dx$.

In fact, $\int_T f(x-b) \,dx$ is strictly increasing in $b$, and we can choose $T$ to be a locally finite union of intervals.
\end{theorem}

\begin{proof}
Since $f$ is absolutely continuous, $f'$ exists almost everywhere, $f' \in L^1$ and $f(x)=\int_{-\infty}^x f'(t) \,dt$ for every $x \in \RR$.
We may suppose that $f$ (and $f'$) is supported in $[-1,1]$ and that $\int_\RR f=1$.

Let $\Phi:\R \to [0,1]$ be an arbitrary $C^1$ function with $\Phi'>0$, and $h:[0,\infty)\to (0,1)$ be an arbitrary decreasing continuous function (which we will specify later). 

By applying Lemma~\ref{l:88} to a sufficiently small function $\de$ 
we obtain a set $T$ such that 
\beq\label{intaf}
\intnorm{\chi_T-\Phi}{1} \le h(|x|) \textrm{ for every } x \in \R.
\eeq
We will complete the proof by proving 
that $f * {\chi_T}$ is strictly increasing. 
As this function is $C^1$ by Lemma \ref{l:conv}, it suffices to show that 
$(f * {\chi_T})'>0$ everywhere. 

Applying Lemma~\ref{l:77} to $\Psi=\chi_T$ and $a=1$, and using (\ref{intaf})
 we obtain
\begin{align*}
(f * {\chi_T})'(x) & 
\ge \min_{[x-1,x+1]} \Phi' - 2\eps - 
K(\eps, f') \intnorm{\chi_T-\Phi}{1} \\
& \ge \min_{[x-1,x+1]} \Phi' - 2\eps  - K(\eps, f') h(|x|). 
\end{align*}

Therefore, choosing $\eps=\eps(x)=1/4 \min_{[x-1,x+1]} \Phi'$,
we see that if we fix $h$ such that
$$h(|x|)\le \frac{\min_{[x-1,x+1]} \Phi'}{4K(\eps(x), f')}$$ 
for every $x\in\R$ then
$$(f*{\chi_T})'(x) \ge 1/4 \min_{[x-1,x+1]} \Phi'>0.$$
\end{proof}

\section{Reconstruction of a function of the form $f(\frac{x}{a}+b)$}

The reconstruction of a magnified copy of a fixed set in $\R^d$ ($d\ge 2$)
will be based on the following result.

\begin{theorem}\label{t2}
Let $f : \RR \to \RR$ be a non-negative not identically zero compactly supported absolutely continuous function. Suppose that
\beq\label{1/3}
\text{there exist } C_1, C_2 \text{ such that } \quad
K(\eps ,f')\le C_1\exp(C_2\eps^{-1/3}) \quad
\text{for every } \eps>0.
\eeq
Then there exists a measurable set $T$ (in fact, a locally finite union of intervals) such that $\int_T f(\frac{x}{a}+b)$ is strictly monotone in $b$ ($b\in\RR$) for every $a\ge 1$.
%
%
\end{theorem}

\br
The theorem does not remain true if we replace $a\ge 1$ by $a>0$.
Indeed, if $b \mapsto \int_{T} f( \frac{x}{a} + b ) \,dx$ is 
strictly monotone for every $a > 0$
then $T$ cannot be of full or zero measure on any interval, so both $T$ and 
its complement has density points on any interval, 
and choosing a small enough $a$ easily shows that monotonicity fails.
\er

Since $\R$ clearly determines $a$, we obtain the following.

\begin{cor}
If $f$ satisfies the conditions of Theorem \ref{t2} then a function of the form 
$f(\frac{x}{a}+b)$ ($a\ge 1$, $b\in\R$) can be reconstructed using two test 
sets, one of which is $\R$.
\end{cor}

\begin{proof}[Proof of Theorem~\ref{t2}]
Since $K(\eps, (cf(x/r+b))')=K(\eps c^{-1}, f') c/r$, property \eqref{1/3} is invariant under linear transformations of $f$. Therefore we may suppose that $\int_{\R}f=1$,
and that $f$ is supported in $[-1,1]$.

Let $f_a(x)=f(x/a)/a$ ($a\ge 1$).
Let $\Phi:\R\to (0,1)$ be a $C^1$ function such that 
$$
\Phi'(x)=\frac{c_1}{|x| \log^2 |x|}
$$ 
when $|x|\ge 2$ (for some positive constant $c_1$) and let $\Phi'>0$ everywhere.
Let $h:[0,\infty)\to(0,1)$ be a decreasing continuous function, 
which we will specify later.

By applying Lemma~\ref{l:88} to a sufficiently small function $\de$
we obtain a set $T$ such that 
\beq\label{defB}
\intnorm{\chi_T-\Phi}{a} \le h(|x-a|)+h(|x+a|).
\eeq
Again, we will complete the proof by proving 
that $(f * {\chi_{T}})'>0$ everywhere.

Applying Lemma~\ref{l:77} to $\Psi=\chi_T$ we get that
\begin{align}\label{e:88}
(f_a * {\chi_{T}})'(x) \ge  \min_{[x-a,x+a]} \Phi' -2\eps/a - K(\eps, f')/a^2 \intnorm{\chi_T - \Phi}{a}
\end{align}
for every $\eps>0$.

We may suppose that $x\ge 0$ as one can deal with the other case similarly.

First let us suppose that $a\ge x/2$, $a\ge 1$.
Then we have 
\beq\label{e:89}
\min_{[x-a,x+a]} \Phi' \ge 
\min\left(\frac{c_1}{|x+a|\log^2|x+a|}, \min_{[-2,2]}\Phi'\right) \ge
\frac{c_2}{3a \log^2(3a)}
\eeq
for some $c_2>0$. 
Using (\ref{defB})
and that $h$ is decreasing we get
\beq\label{e:90}
\intnorm{\chi_T - \Phi}{a} \le h(|x-a|)+h(x+a)\le 2h(0).
\eeq
Choosing $\eps=\frac{c_2}{12\log^2(3a)}$ and combining (\ref{e:88}),
(\ref{e:89}) and (\ref{e:90}) 
we obtain
\begin{align*}
(f_a * {\chi_{T}})'(x) \ge   \frac{c_2}{6a \log^2(3a)} - 
\frac{2h(0)}{a^2} K\left(\frac{c_2}{12\log^2(3a)}, f'\right).
\end{align*}
Using condition (\ref{1/3}) on the magnitude of $K$ we obtain
$$ K\left(\frac{c_2}{12\log^2(3a)}, f'\right) \le C_1\exp(C_2(12\log^2(3a)/c_2)^{1/3}) 
\le \frac{C_3 a}{\log^2(3a)}$$ for some $C_3$,
where the last inequality follows from the fact that the ratio
$$\frac{\exp(C_2(12\log^2(3a)/c_2)^{1/3})}{\frac{a}{\log^2(3a)}}$$ 
tends to $0$ as 
$a\to\infty$ and continuous on $[1,\infty)$, so it is bounded on $[1,\infty)$.
Therefore, if we choose $h(0)$ small enough (for example, $h(0)=c_2/(24C_3)$), 
we have
\begin{align*}
(f_a * {\chi_{T}})'(x) \ge   \frac{c_2}{12a \log^2(3a)} >0
\end{align*}
for every $a\ge 1$ and $x\le 2a$.

Now let us suppose that $1\le a<x/2$.
For some $c_3>0$ we have
\beq\label{e:91}
\min_{[x-a,x+a]} \Phi' \ge \frac{c_3}{2x \log^2(2x)}.
\eeq
Using (\ref{defB})
and that $h$ is decreasing we get
\beq\label{e:92}
\intnorm{\chi_T - \Phi}{a} \le h(x-a)+h(x+a)\le 2h(x/2).
\eeq
Choosing $\eps=\frac{a c_3}{8x \log^2(2x)}$  and combining (\ref{e:88}),
(\ref{e:91}) and (\ref{e:92}) we obtain
\begin{align*}
(f_a * {\chi_{T}})'(x) 
\ge \frac{c_3}{4x \log^2(2x)} - 
K\left(\frac{a c_3}{8x \log^2(2x)}, f'\right) \frac{2h(x/2)}{a^2}.
\end{align*}
Since $K(\eps,f)$ is non-increasing in $\eps$ and $a\ge 1$, we get
$$K\left(\frac{a c_3}{8x \log^2(2x)}, f'\right) \frac{2h(x/2)}{a^2}
\le K\left(\frac{c_3}{8x \log^2(2x)}, f'\right)2h(x/2).$$
Therefore, choosing $h$ such that for every $x\ge 2$ we have
$$
h(x/2) \le \left. \frac{c_3}{16x \log^2(2x)}\middle/ 
K\left(\frac{c_3}{8x \log^2(2x)}, f'\right)\right.,$$
we get that
$$(f_a * {\chi_{T}})'(x) \ge \frac{c_3}{8x \log^2(2x)} >0$$
for every $a\ge 1$ and $x>2a$.
\end{proof}

\begin{remark}
It is not hard to check that one can replace the exponent 
$-1/3$ by $-(1-\delta)$ for
any $\delta>0$ in the condition (\ref{1/3}). 
To obtain this, the function $\Phi$ in the proof has to be chosen so that 
$\Phi'(x)=c_1/(|x|\log^{1+\delta}|x|)$ for $|x|\ge 2$.
We omit the details since we will not need this fact.
\end{remark}

\bd
We say that $x_0 \in \RR$ is a \emph{controlled singularity} of a function $g : \RR \to \RR$ if  $ | g(x )| \le  \frac{1}{|x - x_0|^{1-\delta}}$ in a neighbourhood of $x_0$ for some $\delta>0$, and $g$ is monotone on $(x_0 - \eps, x_0)$ and $(x_0, x_0 + \eps)$ for some $\eps > 0$.
\ed

\begin{lemma}\label{singularity}
If $g$ is measurable, compactly supported, and locally is in $C^1$ except for a finite number of controlled singularities then
$$
\text{there exist } C_1, C_2 \text{ such that } \quad
K(\eps ,g)\le C_1\exp(C_2\eps^{-1/3}) \quad
\text{for every } \eps>0.
$$
\end{lemma}

\begin{proof}
Let us approximate $g$ by $g_n = \min(n, \max(-n, g))$ for a large enough $n$. An easy computation shows that we need $n = C \eps^{- \frac{1 - \delta}{\delta}}$ to achieve $\norm{g - g_n}_1 < \eps$, and then $\Var(g_n) \le C'  \eps^{- \frac{1 - \delta}{\delta}}$. Therefore  $K(\eps, g ) \le C'  \eps^{- \frac{1 - \delta}{\delta}}$ is subexponential and we are done.
\end{proof}

\begin{cor}
In Theorem~\ref{t2} one can replace (\ref{1/3}) by the condition
that $f'$ is locally in $C^1$ except for a finite number of controlled 
singularities.
\end{cor}

\section{Absolute continuity of the Radon transform and
reconstruction of a translate of a fixed set}
\label{s:transl}

\begin{notation}
For a measurable set $E\su\R^d$ $(d\ge 2)$ of finite Lebesgue measure and a unit vector
$\theta\in S^{d-1}$ the \emph{Radon transform in direction $\theta \in S^{d-1}$} is defined
as the measure function of the sections of $E$ in direction $\theta \in S^{d-1}$; that is, 
$$
(R_\theta \chi_E) (r) = 
\la^{d-1} (E \cap \{x \in \RR^d : \langle x,\theta \rangle = r\}),
$$ 
where $\langle\cdot, \cdot \rangle$ denotes scalar product. Note that $R_\theta \chi_E$ is almost everywhere
well defined.
\end{notation}

\begin{theorem}\label{t:abscont}
Suppose that $E\su\R^d$ $(d\ge 2)$ is a bounded measurable set 
with positive Lebesgue measure,
$\theta_1,\ldots,\theta_d\in S^{d-1}$ are linearly independent 
and the Radon transforms $R_{\theta_1} \chi_E, \ldots, R_{\theta_d} \chi_E$ 
are absolutely continuous modulo nullsets. 
Then a translate of $E$ can be reconstructed using $d$ sets.
\end{theorem}

\begin{proof}
We may assume that the Radon transforms are absolutely continuous,
that is, there are no exceptional nullsets, since modifying the functions on
nullsets will have no effect on the following argument. 
By applying Theorem~\ref{t1} to the functions  
$R_{\theta_1} \chi_E, \ldots, R_{\theta_d} \chi_E$ we get measurable test sets
$T_1,\ldots,T_d\su\R$ such that 
\beq\label{intneq}
\int_{T_i} (R_{\theta_i} \chi_E) (x-b) \,dx \neq \int_{T_i} (R_{\theta_{i'}} \chi_E) (x-b') \,dx
\qquad 
(b\neq b',\ i\in\{1,\ldots,d\}).
\eeq
For each $i$ let 
\beq\label{Vi}
V_i= \{a\in\R^d\ :\ \langle a,\theta_i \rangle \in T_i \}.
\eeq
One can easily check that 
$$
\la^d((E+v)\cap V_i)= 
\int_{T_i} (R_{\theta_i} \chi_E) (x-\langle v, \theta_i \rangle) \,dx
$$ 
for any $v\in\R^p$.
Combining this with (\ref{intneq}) we get that $\la^d((E+v)\cap V_i)$
determines $\langle v, \theta_i \rangle$. 
Since $\theta_1,\ldots,\theta_d$ are linearly independent, this implies
that the numbers $\la^d((E+v)\cap V_1),\ldots,\la^d((E+v)\cap V_d)$ 
determine $v$,
which completes the proof.
\end{proof}

\begin{remark}
Since in Theorem~\ref{t1} every test set can be chosen to be a locally finite union
of intervals and the test sets of the above proof are defined by (\ref{Vi}),
each test set of the above theorem (and of all of its corollaries)
can be chosen as a locally finite union of parallel
layers, where by layer we mean a rotated image of a set of the form 
$[a,b]\times \R^{d-1}$.
\end{remark}

The above theorem can clearly be applied to many
geometric objects.

\begin{cor}
\begin{enumerate}
\item
A ball of fixed radius 
in $\R^d$ ($d\ge 1$) can be reconstructed using $d$ sets; that is, 
for any $r$ there exist measurable sets $T_1, \dots,  T_d \su \RR^d$ such that if $x \neq x'$ then $\la^d (B(x,r) \cap T_i) \neq \la^d (B(x',r) \cap T_i)$ for some $i \in \{ 1, \dots, d \}$.
\item
Let $E$ be a (not necessarily convex) polytope in $\R^d$ ($d\ge 2$). 
Then a translate of $E$ can be reconstructed using $d$ test sets.
\end{enumerate}
\end{cor}

\begin{proof}
In Theorem~\ref{t:unitint} we already proved the case $d=1$ of (1).

Now let $d \ge 2$. 
If $B$ is a fixed ball then $R_\theta \chi_B$ is clearly absolutely continuous
for every $\theta$.
If $E$ is a polytope in $\R^d$ then $R_\theta \chi_E$ is absolutely 
continuous for any $\theta$ which is not orthogonal to any face of $E$.
Therefore in both cases Theorem~\ref{t:abscont} can be applied.
\end{proof}

In the remaining part of this section
in order to apply Theorem~\ref{t:abscont} for a more general set $E\in\R^d$,
we try to find 
many directions $\theta\in S^{d-1}$ for 
which $R_\theta \chi_E$ is absolutely continuous modulo a nullset. 

To get a general positive result for $d\ge 3$ we use Fourier transforms.
Denote the Fourier transform of a function $f$ by $\hat{f}$.

\bl
\lab{l:weak}
Let $f : \RR \to \RR$ be a compactly supported $L^2$ function. If $\ r\hat{f}(r) \in L^2$ then $f$ is absolutely continuous modulo a nullset and $f' \in L^2$. 
\el

\bp
Recall that an $L^1$ function agrees with an absolutely continuous function almost everywhere if and only if
its weak derivative is an $L^1$ function. (Indeed, this is the well-known fact that the Sobolev space $W^{1,1}$ is the class of absolutely continuous function modulo nullsets, see \cite[Corollary 7.14.]{Le}.)

Therefore it suffices to prove that the function
\[
f^* (r) = \widehat{ - 2 \pi i r \hat{f} (-r)}  \quad (r \in \RR)
\]
is in $L^1$, it is the weak derivative of $f$, and that $f^* \in L^2$. 
Clearly, $f^* \in L^2$ follows from the assumption $\ r\hat{f}(r) \in L^2$. 
Let $\phi$ be an arbitrary $C^\infty$ function of compact support. Using the Parseval Formula
twice as well as $\widehat{\psi'} (r)= 2 \pi i r \hat{\psi} (r)$ and $\hat{\hat{g}} (r) = g(-r)$ we obtain
\[
\int_\RR f^* \phi = \int_\RR  \widehat{-2 \pi i r \hat{f} (-r)} \overline{\overline{\phi(r)}} \,dr = \int_\RR  2 \pi i r \hat{f} (r) \overline{\widehat{\overline{\phi(r)}}} \,dr =
\]
\[
- \int_\RR \hat{f}(r) \overline{2 \pi i r  \widehat{\overline{\phi(r)}}} \,dr =
- \int_\RR \hat{f}(r) \overline{\widehat{\overline{\phi}'(r)}} \,dr =
- \int_\RR f(r) \overline{\overline{\phi}'(r)} \,dr = - \int_\RR f \phi',
\]
which yields that $f^*$ is the weak derivative of $f$. But it is easy to see that the support of the weak derivative of $f$ is contained in $\supp(f)$, hence $f^*$ is a compactly supported $L^2$ function, therefore it is in $L^1$, which concludes the proof.
\ep

\begin{lemma}\label{l:fourier}
Let $K\su\R^d$ ($d\ge 2$) be a bounded measurable 
set of positive Lebesgue measure.
Then for almost every $\theta\in S^{d-1}$ we have 
$$
\int_\R |\widehat{R_\theta \chi_K} (r) |^2 |r|^p \,dr<\infty 
\textrm{ for any } 0 \le p \le d-1.
$$
\end{lemma}

\begin{proof}
By Plancherel Theorem  $\int_{\RR^d} |\hat{{\chi_K}}|^2 = \int_{\RR^d} {\chi_K}^2 <\infty$. Therefore, using polar coordinates, for almost every direction $\theta$,
\beq
\lab{e:Lp}
\int_\R |\widehat{{\chi_K}} (r\theta) |^2 |r|^{d-1} \,dr<\infty.
\eeq
Fix such a $\theta$. 
Since $\chi_K\in L^1$, the function $\widehat{{\chi_K}}$
is bounded, so (\ref{e:Lp}) implies that 
\beq\label{e:L2}
\int_\R |\widehat{{\chi_K}} (r\theta) |^2 |r|^p \,dr<\infty
\eeq
for any $p \le d-1$.
An easy computation shows the well-known fact that 
$\widehat{R_\theta \chi_K}(r)=\widehat{{\chi_K}}(r\theta)$, so we are done.
\end{proof}

\begin{theorem}\label{t:abscontsections}
Let $K\su\R^d$ ($d\ge 3$) be a bounded measurable 
set. Then the Radon transform of $\chi_K$ in direction $\theta$, that is, 
$$
(R_\theta \chi_K) (r) = 
\la^{d-1} (K \cap \{x \in \RR^d : \langle x,\theta \rangle = r\})
$$
is absolutely continuous for almost every $\theta\in S^{d-1}$.
\end{theorem}

\begin{proof}
Since $d\ge 3$ we can apply Lemma~\ref{l:fourier} for $p=2$ to get that
$r\widehat{R_\theta \chi_K} (r) \in L^2$ for almost every $\theta\in S^{d-1}$.
Hence Lemma \ref{l:weak} applied to $R_\theta \chi_K$
gives that $R_\theta \chi_K$ is absolutely continuous modulo a nullset for almost every $\theta \in S^{d-1}$. By \cite[Corollary 2]{OS}, $R_\theta \chi_K$ is continuous for almost every $\theta$, which completes the proof.
\end{proof}

\begin{remark}
In fact, we do not need that the functions $R_\theta \chi_K$ are actually continuous for almost every $\theta$. We will only use that they are absolutely continuous modulo nullsets for our applications.
\end{remark}

Combining Theorems~\ref{t:abscont}~and~\ref{t:abscontsections} we get
the following.

\begin{cor}\label{c:posmeasure}
Let $d\ge 3$ and let $E\subset\R^d$ be a bounded set of positive Lebesgue
measure. Then a translate of $E$ can be reconstructed using $d$ sets; that is,
there are measurable sets $T_1, \dots,  T_d \su \RR^d$ such that if $x \neq
x'$ then $\la^d ((E+x) \cap T_i) \neq \la^d ((E+x') \cap T_i)$ for some $i \in
\{ 1, \dots, d \}$. 
\end{cor}

We do not know whether Corollary~\ref{c:posmeasure} holds for $d=1$ and $d=2$.
Our method clearly cannot work for $d=1$.
The next theorem shows that we cannot obtain
Corollary~\ref{c:posmeasure} for $d=2$ the same way,
since Theorem~\ref{t:abscontsections} fails in $\R^2$ in the following strong
sense.

\begin{theorem}
There exists a bounded measurable set $K$ in $\R^2$ such that for every 
direction $\theta$ the Radon transform in direction $\theta$ 
does not agree almost everywhere with a continuous function.
\end{theorem}
\begin{proof}
We call a planar set a Besicovitch set if it is measurable and it contains a unit line segment in every direction.
It is well known that there exists a compact Besicovitch set of measure zero, let $A$ be such a set. For each $n\ge 1$, let $A_n$ be an open neighbourhood of $A$ of Lebesgue measure at most $1/2^n$. Let $p_i$ be a sequence of points dense in the unit disc. Take $K=\bigcup_{n=1}^\infty A_n+p_n$. Then the measure of $K$ is at most $1$. Since $A$ contains a unit line segment in every direction, for every $\theta$ the Radon transform (the measure function of the sections) $(R_\theta \chi_K)(x)\ge 1$ for $x\in U_\theta$ where $U_\theta$ is a dense open subset of an interval of length $2$. Suppose that $R_\theta \chi_K$ agrees with a continuous function almost everywhere. Then $(R_\theta \chi_K)(x) \ge 1$ on an interval of length $2$ (almost everywhere), thus $\int R_\theta \chi_K\ge 2$, which contradicts the fact that the measure of $K$ is at most $1$.
\end{proof}

If we require only the continuity of $R_\theta \chi_K$ 
then it is enough to assume that the boundary of $K$ has Hausdorff dimension less
than $2$. Since we do not need this result, we only sketch the proof.

\begin{theorem}\label{Besicovitch}
Let $K$ be a bounded Borel
set in $\R^2$ such that $\partial K$ has Hausdorff dimension less than $2$.
Then the Radon transform of $\chi_K$ in direction $\theta$ (that is, the measure function of the sections of $K$ in direction $\theta$) is continuous for almost every $\theta$.
\end{theorem}
\begin{proof} (Sketch)
Let $\clK$ denote the closure of $K$. If $R_\theta \chi_\clK \neq R_\theta \chi_K$, then there exists a line perpendicular to $\theta$ which intersects $\partial K$ in a set of positive (one-dimensional Lebesgue) measure.

Since $\overline{K}$ is compact, $R_\theta \chi_\clK$ is easily seen to be upper semi-continuous for every $\theta$.
%
If $\partial K$ has zero one-dimensional Lebesgue measure on the line $\{x\in\R^2 : \langle x,\theta \rangle=a\}$, then $R_\theta \chi_K$ is lower semi-continuous at $a$.

From these observations it follows that if $R_\theta \chi_K$ is not continuous, then there exists a line perpendicular to $\theta$ which intersects $\partial K$ in a set of positive (one-dimensional Lebesgue) measure.

Now suppose to the contrary that $R_\theta \chi_K$ is not continuous in positively many directions. Then
there are positively many directions in which there are lines which
intersect $\partial K$ in a set of positive (one-dimensional Lebesgue)
measure. It is well known that this implies that $\partial K$ has Hausdorff
dimension $2$. (This is a slight generalization of the fact that every planar
Besicovitch set must have Hausdorff dimension $2$, cf. \cite[Proposition 1.5 and Lemma 1.6]{Wo}.)
\end{proof}

For absolute continuity of $R_\theta \chi_K$ we need to assume more. Recall
the definition of Hausdorff measure (and dimension) from \cite{Fa} or
\cite{Ma} and also that by the length of a set $A$ we mean $\iH^1(A)$, that
is, the 1-dimensional Hausdorff measure of $A$. A set is
rectifiable if it can be covered by countably many Lipschitz curves and an
$\iH^1$-null set. A set is purely unrectifiable if it intersects every
rectifiable set (equivalently, every Lipschitz curve) in an $\iH^1$-null set.

\begin{theorem}\label{t:rectbound}
Let $K$ be a compact set in $\R^2$ with rectifiable boundary
of finite length. 
Then $R_\theta \chi_K$ is absolutely continuous for all but 
countably many $\theta$.
\end{theorem}

\begin{proof}
Let $\mu$ denote the $1$-dimensional Hausdorff measure restricted 
to $\partial K$, and let $\mu_\theta$ denote 
the projection of $\mu$ to the line in direction $\theta$. 

\begin{lemma}\label{abscontlength_new}
For any interval $[x,y]$ and direction $\theta$ we have
$$
\abs{(R_\theta \chi_K)(y) - (R_\theta \chi_K)(x)} \le \mu_\theta([x,y]).
$$
\end{lemma}

\begin{proof}
We can clearly assume that $\theta$ is vertical. 
Then
$\abs{(R_\theta \chi_K)(y) - (R_\theta \chi_K)(x)}$ is the difference of
the measures 
of $(\{x\}\times \R) \cap K$ and $(\{y\}\times \R) \cap K$ (two vertical lines
intersected with $K$).  Clearly, $\partial K$ must intersect those horizontal
segments $[(x, t), (y, t)]$ for which $(x, t) \in (\{x\}\times \R) \cap
K$ but $(y, t) \notin (\{y\}\times \R) \cap K$ or vice versa.  The measure
of these $t$ is at least $\abs{(R_\theta \chi_K)(y) - (R_\theta \chi_K)(x)}$, thus
the projection of $([x, y]\times \R) \cap \partial K$ on the vertical axis
has Lebesgue measure at least $\abs{(R_\theta \chi_K)(y) -
  (R_\theta \chi_K)(x)}$. 
Since projections do not increase Hausdorff measure and Lebesgue measure on a 
line agrees with the $1$-dimensional Hausdorff measure, this implies that
$$
\abs{(R_\theta \chi_K)(y) - (R_\theta \chi_K)(x)} \le 
\cah^1(([x,y]\times \R)\cap \partial K) = \mu ([x,y]\times\R)=
\mu_\theta([x,y]).
$$
\end{proof}

\begin{lemma}\label{pu222}
For every direction $\theta$, 
if $\mu_\theta$ is absolutely continuous with respect to the Lebesgue measure
then the function $R_\theta \chi_K$ is absolutely continuous.
\end{lemma}
\begin{proof}
Suppose $\theta$ is such that $\mu_\theta$ is absolutely continuous with respect to the Lebesgue measure.
Since $\partial K$ has finite length, $\mu$ and $\mu_\theta$ are finite measures.
Therefore the Radon--Nikodym derivative of $\mu_\theta$ is in $L^1$, and thus
the function $x\mapsto \mu_\theta((-\infty, x])$ is absolutely continuous. 
Recall that a real function $f$ is absolutely continuous if and only if 
for every $\eps>0$ there exists $\delta>0$ such that for every finite system of disjoint intervals $[x_j, y_j]$ satisfying $\sum_j \abs{y_j -x_j} < \delta$ we have $\sum_j \abs{f(y_j) - f(x_j)} <\eps$.
Thus, by Lemma~\ref{abscontlength_new}, the absolute continuity of 
the function $x\mapsto \mu_\theta((-\infty, x])$
 implies the absolute continuity of the function 
$R_\theta \chi_K$.
\end{proof}
To finish the proof of Theorem~\ref{t:rectbound} we have to show that $\mu_\theta$ is absolutely continuous for all but countably many $\theta$. Suppose to the contrary that for uncountably many $\theta$, there are Borel sets $A_\theta \subset\partial K$ with $\cah^1(A_\theta)>0$, such that the projection of $A_\theta$ to the line in direction $\theta$ has Lebesgue measure zero. Then there is an $\eps>0$ such that $\cah^1(A_\theta)>\eps$ for uncountably many $\theta$. As $\cah^1(\partial K)<\infty$, there are $\theta$ and $\theta'$ such that $\cah^1(A_\theta \cap A_{\theta'})>0$. 
Therefore
$A_\theta \cap A_{\theta'}$ is a rectifiable set of positive length, and there 
are two directions in which the projection has Lebesgue measure zero.
It is well-known (see e.g. in \cite[18.10 (4)]{Ma}) that this is impossible.
\end{proof}

\begin{remark}
The condition that the boundary of $K$ has finite length cannot be omitted in
Theorem~\ref{t:rectbound}.
There exists a compact set with rectifiable boundary of $\sigma$-finite
length so that for \emph{every} direction the Radon transform of its 
characteristic function is not of bounded variation, hence not absolutely
continuous.
%
We sketch the random construction. 

Let $A_N$ be the random compact set we obtain by decomposing the unit square 
into $N\times N$ many squares of side-length $1/N$ and keeping each of them independently with
probability $1/2$. It can be shown that the total variation of $R_\theta \chi_{A_N}$ is at least $N^{1/2-\eps}$ for every $\theta$, with probability tending to $1$ as $N\to \infty$.

Now let
$$A=\{(0,0)\} \cup \bigcup_{k=0}^\infty \left(\alpha_k A_{N_k} + \frac{1}{2^k}\right), $$
where $\alpha_k\to 0$ and $N_k\to\infty$ sufficiently rapidly. Then $A$ is compact and $\partial A$ is rectifiable of $\sigma$-finite length. 
It can be shown that, with probability $1$, $R_\theta \chi_{A}$ is not of bounded variation for any $\theta$.

\end{remark}

\begin{cor}\label{c:rectbound}
Let $E$ be a bounded measurable set in $\R^2$ with positive Lebesgue measure
and rectifiable boundary of finite length. 
Then a translate of $E$ can be reconstructed using $2$ test sets.
\end{cor}

\begin{proof}
We apply Theorem~\ref{t:rectbound} for $K=\overline E$.
Since $K\sm E \su \partial E$ has Lebesgue measure zero,
$R_\theta \chi_E$ equals almost everywhere to $R_\theta \chi_K$ for every
$\theta$, so Theorem~\ref{t:abscont} can be applied. 
\end{proof}

From Theorem~\ref{t:finiteint} and
Corollaries~\ref{c:posmeasure}~and~\ref{c:rectbound} we get the following
in any dimension.

\begin{cor}
A translate of a fixed finite union of bounded
convex sets in $\R^d$ ($d=1,2,\ldots$) can 
be reconstructed using $d$ test sets.
\end{cor}

\section{Reconstruction of a magnified copy of a fixed set}
\label{s:magn} 

The first part of this section is analogous to the first part of the 
previous section but here the results follow from Theorem~\ref{t2} 
instead of Theorem~\ref{t1}.
  
\begin{theorem}\label{t:magn}
Let $E\su\R^d$ $(d\ge 2)$ be a bounded measurable set 
with positive Lebesgue measure. 
Suppose that
$\theta_1,\ldots,\theta_d\in S^{d-1}$ are linearly independent 
such that for each $i=1,\ldots,d$ the Radon transform of $\chi_E$ in direction $\theta_i$
is absolutely continuous modulo a nullset and 
there exist $C_1$, $C_2$ such that $K(\eps ,(R_{\theta_i} \chi_E)')\le C_1\exp(C_2\eps^{-1/3})$ for every $\eps>0$.

Then a set of the form $rE+x$, where $r\ge 1$ and $x\in\R^d$, can be
reconstructed using $d+1$ test sets.
\end{theorem}

\begin{proof}
As in the proof of Theorem~\ref{t:abscont}, we may assume that the Radon transforms are actually absolutely continuous.
By applying Theorem~\ref{t2} to the functions  
$R_{\theta_1} \chi_E,\ldots,R_{\theta_d} \chi_E$ we get measurable sets
$T_1,\ldots,T_d\su\R$ such that for each $i$ and $r\ge 1$ the integral
$\int_{T_i} (R_{\theta_i} \chi_E)(\frac{x}{r}-b) \,dx$ determines $b$.
For each $i$ let 
$V_i= \{a\in\R^d\ :\ \langle a,\theta_i \rangle \in T_i \}$.
One can easily check that
$$
\la^d((rE+v)\cap V_i)= 
r^{d-1} \int_{T_i} (R_{\theta_i} \chi_E)\left(\frac{x-\langle v, \theta_i \rangle}{r}\right) \,dx
$$ 
for any $v\in\R^p$.

Therefore, for any $r\ge 1$ the 
numbers $\la^d((rE+v)\cap V_1),\ldots,\la^d((rE+v)\cap V_d)$ 
determine $v$.
Let $V_{d+1}=\R^d$. Then $\la^d((rE+v)\cap V_{d+1})$ clearly 
determines $r$, 
which completes the proof.
\end{proof}

\begin{remark}
Since in Theorem~\ref{t2} 
the test set that determines $b$
can be chosen to be a locally finite union
of intervals,
we obtained that 
each of the first $d$ test sets of the above theorem 
(and of all of its corollaries)
can be chosen as finite union of parallel
layers (where by layer we mean a rotated image of a set of the form 
$[a,b]\times \R^{d-1}$) and one test set as $\R^d$.
\end{remark}

The above theorem can clearly be applied to many geometric objects.

\begin{cor}\label{c:magnified}
\begin{enumerate}
\item\label{largeball}
A ball of radius at least $1$ in $\R^d$ ($d\ge 2$) can be reconstructed 
using $d+1$ sets; that is, 
there are measurable sets $T_1, \dots,  T_{d+1} \su \RR^d$ such that if $(x,r) \neq (x', r')$, $x, x' \in \RR^d$, $r, r' \ge 1$ then $\la^d (B(x,r) \cap T_i) \neq \la^d (B(x',r') \cap T_i)$ for some $i \in \{ 1, \dots, d+1 \}$.
\item\label{polytope}
Let $E$ be a (not necessarily convex) polytope in $\R^d$ ($d\ge 2$). 
Then a magnified copy $rE+x$, where $r\ge 1$ and $x\in \R^d$  
can be reconstructed using $d+1$ test sets.
\end{enumerate}
\end{cor}

\begin{proof}
It is easy to check that the assumptions of Lemma~\ref{singularity} hold for 
$(R_\theta \chi_E)'$ if $E$ is the unit ball or if
$E$ is a polytope and $\theta$ is not orthogonal to any of its faces. 
Thus we can apply Theorem~\ref{t:magn} to $E$ in both cases.
\end{proof}

\begin{remark}
By Theorem~\ref{t:interval} the above corollary does not hold for $d=1$.
\end{remark}

In the remaining part of this section we check the condition
of Theorem~\ref{t:magn} for more general sets.
For the most general theorem we need the following result concerning the Radon transforms, which
we can only prove for $d\ge 4$.

\begin{theorem}\label{t:sectionsformagn}
If $d\ge 4$ then for any bounded
measurable set $E\su\R^d$ the Radon transform $R_\theta \chi_E$
(that is, the measure function of the sections of $E$ in direction $\theta$) 
is absolutely continuous for almost every $\theta\in S^{d-1}$,
and 
$K(\eps, (R_\theta \chi_E)')\le C_\theta\eps^{-2}$,
where $C_\theta$ depends only on $E$ and $\theta$.
\end{theorem}

\begin{proof}
By Theorem~\ref{t:abscontsections}, $R_\theta \chi_E$ is absolutely continuous for almost every $\theta$.
Applying Lemma~\ref{l:fourier} for $p=3$ and $p=2$ 
we get that for almost every $\theta$, 
\beq\label{e:cubic}
\int |x|^3 |\widehat{R_\theta \chi_E}|^2(x) <\infty 
\qquad \text{and} \qquad
\int |x|^2 |\widehat{R_\theta \chi_E}|^2(x) <\infty.
\eeq
Fix such a $\theta$ and put $f=R_\theta \chi_E$. 
Thus, denoting the weak derivative of $f$ by $f'$,
\beq\label{e:deriv}
\int |x| |\widehat{f'}|^2(x) <\infty
\qquad \text{and} \qquad
\int |\widehat{f'}|^2 <\infty.
\eeq
We may assume that $f$ (and thus $f'$) is supported in $[0,1]$.
Note that $f'$ is in $L^1 \cap L^2$.
Let $g_R(x)=(\chi_{[-R,R]} \widehat{f'})\widehat{\;}\,(-x)$. 
Thus $g_R$ is $C^\infty$ and $\widehat{g_R}=\chi_{[-R,R]} \widehat{f'}$.
We will approximate $f'$ by $\chi_{[0,1]}g_R$ to get a bound on $K(\eps, f')$.

We have
\begin{align}\label{L1benk}
\norm{\chi_{[0,1]}g_R - f'}_1 & = \norm{\chi_{[0,1]}(g_R - f')}_1 
\le \norm{\chi_{[0,1]}(g_R - f')}_2
\le  \norm{g_R - f'}_2 \nonumber \\
& = \norm{\widehat{g_R} - \widehat{f'}}_2 = \norm{(1-\chi_{[-R,R]}) \widehat{f'}}_2 \\
& \le \norm{|x|^{1/2} R^{-1/2} \widehat{f'} (x) }_2 = R^{-1/2} \left(\int |x| |\widehat{f'}|^2(x) \,dx \right)^{1/2}. \nonumber 
\end{align}
We have to bound the total variation of $\chi_{[0,1]}g_R$. 
We will do this in two steps. First,
\begin{align}\label{boundvariation1}
\norm{\chi_{[0,1]} g_R'}_1 & \le \norm{\chi_{[0,1]} g_R'}_2 \le \norm{g_R'}_2  
= \norm{\widehat{g_R'}}_2 =
2\pi \norm{x\widehat{g_R} (x) }_2  \le 2\pi  \norm{x\chi_{[-R,R]}
  \widehat{f'}(x) }_2 \\
& \le 2\pi \norm{|x|^{1/2} R^{1/2} \widehat{f'} (x) }_2  = 2\pi R^{1/2} \left(\int |x| |\widehat{f'}|^2(x) \,dx \right)^{1/2}. \nonumber
\end{align}
Second,
\begin{align}\label{boundvariation2}
\norm{g_R}_\infty & \le \norm{\widehat{g_R}}_1 = \norm{\chi_{[-R,R]} \widehat{f'}}_1 \le 2R \norm{\widehat{f'}}_\infty
\le 2R \norm{f'}_1 \\
& \le 2R \norm{f'}_2  = 2R \norm{\widehat{f'}}_2 = 4\pi R
\norm{x\widehat{f} (x) }_2 \le 4\pi R \left( \norm{|x|^{3/2} \widehat{f} (x)
}_2^2 + \norm{\hat{f}}^2_2 \right)^{1/2},
\nonumber
\end{align}
where in the last inequality we used that $x^2\le 1$ on $[-1,1]$ and
$x^2\le |x|^3$ outside $[-1,1]$.

Combining \eqref{e:cubic}, \eqref{e:deriv},
\eqref{boundvariation1} and \eqref{boundvariation2} gives that the total variation of $\chi_{[0,1]}g_R$ is at most
$$2 \norm{g_R}_\infty + \norm{\chi_{[0,1]} g_R'}_1 \le c_1 R$$
for some finite positive constant $c_1$ (depending on $f$). Comparing this to \eqref{L1benk} gives that
$$K(c_2 R^{-1/2}, f')\le c_1R$$
for some $c_2>0$, and thus $K(\eps, f') \le C \eps^{-2}$.
\end{proof}

\begin{remark}
The proof of the previous theorem is simpler for $d\ge 5$. In these dimensions
we obtain from Lemma~\ref{l:fourier} that $r^2 \widehat{R_\theta \chi_E}(r) \in
L^2$ and $r \widehat{R_\theta \chi_E}(r) \in L^2$ for almost every $\theta$. 
Then by Lemma \ref{l:weak}, $R_\theta \chi_E$ is absolutely continuous (we may
ignore the nullset) and $(R_\theta \chi_E)' \in L^2$. It is easy to see that the usual proof of
the formula $\widehat{f'} (r) = 2 \pi i r \hat{f}$ works if we only assume that
$f$ is absolutely continuous modulo a nullset. Hence
$r((R_\theta \chi_E)')\widehat{\;} (r) = r(2 \pi i r \widehat{R_\theta \chi_E}) \in L^2$, so
a second application of Lemma \ref{l:weak} yields that $(R_\theta \chi_E)'$ is
absolutely continuous (ignoring the nullset again).
Therefore $K(\eps, (R_\theta \chi_E)') \le \Var((R_\theta \chi_E)')$ is bounded in this case.
\end{remark}

Theorems~\ref{t:magn}~and~\ref{t:sectionsformagn} immediately imply 
the following.

\begin{cor}\label{c:genmagn}
Let $d\ge 4$ and let $E\subset\R^d$ be a bounded set of positive Lebesgue measure. Then a set of the form $rE+x$, where $r\ge 1$ and $x\in\R^d$, can be reconstructed using $d+1$ sets; that is, there are measurable sets $T_1, \dots,  T_{d+1} \su \RR^d$ such that if $(x,r) \neq (x', r')$, $x,x' \in \RR^d$, $r, r' \ge 1$ then $\la^d ((rE+x) \cap T_i) \neq \la^d ((r'E+x') \cap T_i)$ for some $i \in \{ 1, \dots, d+1 \}$.
\end{cor}

In the remaining part of this section we show that for convex $E$ the above
result also holds for $d\ge 2$. 

\begin{lemma}\label{brunn}
Let $d\ge 2$ and let $E\subset \R^d$ be a bounded convex set of non-empty interior. 
Then for every direction $\theta$ the function 
$(R_\theta \chi_E)^{1/(d-1)}$ is concave on its support,
where $R_\theta \chi_E$ denotes the 
Radon transform of $\chi_E$ in direction $\theta$.
\end{lemma}

\begin{proof}
Let $\theta$ be an arbitrary direction. 
Let $E_x = \{a\in E : \sk{a,\theta}=x\}$ for $x\in\R$. By convexity of $E$ we have
\beq\label{convexity}
(1-t) E_x + t E_y \subset E_{(1-t)x+ty},
\eeq
where $0\le t\le 1$.
Applying the Brunn--Minkowski inequality for $(d-1)$-dimensional convex bodies gives
$$\lambda^{d-1}((1-t) E_x + t E_y)^{1/(d-1)} \ge (1-t)\lambda^{d-1}(E_x)^{1/(d-1)}  +  t\lambda^{d-1}(E_y)^{1/(d-1)},$$
supposing that both $E_x$ and $E_y$ are non-empty.
Combining this with \eqref{convexity} gives
$$\lambda^{d-1}(E_{(1-t)x+ty})^{1/(d-1)} \ge (1-t)\lambda^{d-1}(E_x)^{1/(d-1)}  +  t\lambda^{d-1}(E_y)^{1/(d-1)}.$$
That is, $(R_\theta \chi_E)((1-t)x+ty )^{1/(d-1)} \ge (1-t) (R_\theta \chi_E)(x )^{1/(d-1)} + t (R_\theta \chi_E)(y)^{1/(d-1)}$ whenever $x$ and $y$ are in the support of $R_\theta \chi_E$.
\end{proof}

\begin{lemma}\label{konvex1irany}
Let $d\ge 2$ and let $E\subset \R^d$ be a bounded convex set of non-empty interior. Let $x,y\in \overline{E}$ have maximal distance among all pairs of points 
of the closure $\overline{E}$ of $E$, and let $\theta$ be the direction of $xy$. Then $R_\theta \chi_E$ is absolutely continuous and satisfies $K(\eps, (R_\theta \chi_E)') = O(1/\eps^{1/(d-1)})$.
\end{lemma}
\begin{proof}
We may suppose without loss of generality that the distance of $x$ and $y$ is $1$, and the support of the Radon transform $f = R_\theta \chi_E$ is $[0,1]$. 
The set $E$ is contained in the balls of unit radius centered at $x$ and
$y$. This implies that 
\beq\label{i1}
f(t)\le C t^{(d-1)/2} \text{ and } f(1-t) \le C t^{(d-1)/2} \qquad (0\le t\le 1).
\eeq

Lemma~\ref{brunn} implies that $g=f^{1/(d-1)}$ is concave on its support. 
Let $g'$ denote the everywhere existing right-hand derivative of $g$, which is also a weak derivative of $g$.
Concavity and $g(0)=g(1)=0$ imply that
\beq\label{i2}
g'(t)\le \frac{g(t)}{t} \text{ and }g'(1-t)\ge -g(1-t)/t \qquad (0<t\le 1).
\eeq
If we combine this with \eqref{i1} we obtain
\beq\label{i2.5}
|g'(t)| \le \frac{C'}{\sqrt{t}} \text{ \ and \ } |g'(1-t)| \le \frac{C'}{\sqrt{t}} \qquad (0<t<1).
\eeq

The formula 
$$f'(t) = (g^{d-1})'(t) = (d-1) g^{d-2}(t) g'(t) \qquad (0\le t\le 1)$$
implies that $f$ has an everywhere existing right-hand derivative, we denote it by $f'$. Clearly $f'$ is also a weak derivative of $f$.
Thus by \eqref{i2},
\beq\label{i4}
f'(t) \le (d-1) g^{d-2}(t) \frac{g(t)}{t} \le (d-1) \frac{f(t)}{t} \qquad (0<t\le 1).
\eeq

Let us fix a small $\eps>0$. Let $h:\R \to \R$ be defined as $h(x)=f'(x)$ if $\eps\le x \le 1-\eps$, and $h(x)=0$ otherwise. 
The function $g$ is concave, nonnegative, $g(0)=g(1)=0$, so
$g$ is monotone in $[0,\eps]$ and in $[1-\eps,1]$ if $\eps$ is small enough. 
Hence $f=g^{d-2}$ is also monotone in the same intervals, thus 
$$\int_0^\eps |f'(t)| + \int_{1-\eps}^1 |f'(t)| = f(\eps) + f(1-\eps).$$
Using this and \eqref{i1} we obtain
\beq\label{i5}
\norm{f'-h}_1 =  f(\eps) + f(1-\eps) \le 2C \eps^{(d-1)/2}.
\eeq


%
We have to give an upper bound for $\Var(h)$. We will use the inequality
$$\Var(f_1 f_2)\le \Var(f_1)\sup |f_2| + \Var(f_2)\sup |f_1|.$$
Writing $m$ for $\max_{[0,1]} g$,
\begin{align}
Var_{[\eps,1-\eps]}(f') & = Var_{[\eps,1-\eps]}((d-1) g^{d-2} g') \nonumber \\
 & \le (d-1) \left( Var_{[\eps,1-\eps]}(g') m^{d-2} 
+  Var_{[\eps,1-\eps]}(g^{d-2})  \max_{[\eps,1-\eps]}|g'| \right) \nonumber \\
& \le (d-1) \left( Var_{[\eps,1-\eps]}(g') m^{d-2} 
+  2m^{d-2} \max_{[\eps,1-\eps]}|g'| \right) \label{i6}
\end{align}
where $Var_{[0,1]}(g^{d-2}) = 2m^{d-2}$ (if $d\ge 3$) follows from $g$ being concave. Note that \eqref{i6} holds for $d=2$ as well.
As $g'$ is nonincreasing, 
$Var_{[\eps,1-\eps]}(g') = |g'(1-\eps)-g'(\eps)|=|g'(\eps)| + |g'(1-\eps)|$
and $\max_{[\eps,1-\eps]}|g'|=\max(|g'(\eps)|,|g'(1-\eps)|)$. 
Therefore \eqref{i6} and \eqref{i2.5} implies
\begin{align*}
Var_{[\eps,1-\eps]}(f') & \le  (d-1) 4C' m^{d-2}/\sqrt{\eps}.
\end{align*}
Thus
\begin{align}
\Var(h) & = h(\eps)+ h(1-\eps) + Var_{[\eps,1-\eps]}(h) \nonumber\\
& = |f'(\eps)| + |f'(1-\eps)| + Var_{[\eps, 1-\eps]}(f') \nonumber \\
& \le  |f'(\eps)| + |f'(1-\eps)| + (d-1) 4C' m^{d-2}/\sqrt{\eps}.
\label{i6.6}
\end{align}
Using \eqref{i4}, \eqref{i1} and $d\ge 2$ we get that
both $|f'(\eps)|$ and $|f'(1-\eps)|$ are at most 
$(d-1)\eps^{(d-3)/2} \le (d-1)/\sqrt{\eps}$.

Using this, \eqref{i6.6} and \eqref{i1} we obtain
\begin{align}\label{i7}
\Var(h) & \le |f'(\eps)| + |f'(1-\eps)| + (d-1) 4C' m^{d-2}/\sqrt{\eps} 
\le C'' /\sqrt{\eps},
\end{align}
where $C''$ depends on $m$, but not on $\eps$.
Combining \eqref{i5} and \eqref{i7} 
and setting $\delta =  2C \eps^{(d-1)/2}$
give that $K(\delta, f') \le C''' \delta^{-1/(d-1)}$ if $\delta>0$ is small enough.
\end{proof}

\begin{theorem}\label{konvexsuru}
Let $d\ge 2$ and let $E\subset \R^d$ be a bounded convex set of non-empty interior. There exist a dense set of directions $\theta$ for which the Radon transforms $R_\theta \chi_E$ are absolutely continuous and satisfy $K(\eps, (R_\theta \chi_E)')=O(1/\eps^{1/(d-1)})$.
\end{theorem}

\begin{proof}
We will find an appropriate $\theta$ arbitrarily close to the vertical direction. This will prove the theorem. Let $\delta>0$ be small. Let $\Phi$ be the linear transformation which maps $(x_1, \ldots, x_d)$ to $(\delta x_1, \ldots, \delta x_{d-1}, x_d)$. (We call the $x_d$ coordinate direction vertical.) Let $E_\delta = \Phi(E)$.

Suppose that the projection of $E$ to the vertical axis has diameter $1$. 
Let $x,y\in \overline{E_\delta}$ be points which have maximal distance among all pairs of points of $\overline{E_\delta}$. Their distance is at least $1$. For some constant $C$ (depending on $E$ only), $E_\delta$ is contained in a right circular cylinder in vertical position of radius $C\delta$ and height $1$. This implies that the distance of the direction of $xy$ to the vertical direction is at most $C'\delta$.

Let us apply Lemma~\ref{konvex1irany} to $E_\delta$. We obtain a direction $\theta$ which is $C'\delta$ close to vertical such that the Radon transform $R_\theta \chi_{E_\delta}$ has the right properties.
Consider $E_\delta$ and the hyperplanes which are orthogonal to $\theta$. If we apply $\Phi^{-1}$ to them, we get $E$ and the new hyperplanes will be orthogonal to a direction which is $C''\delta^2$ close to vertical---in fact, they are orthogonal to $\Phi^*(\theta)=\Phi(\theta)$ as $\Phi$ is self-adjoint.
Since $\Phi$ is a linear map, the Radon transform $R_{\Phi(\theta)} \chi_E$ can be obtained from $R_\theta \chi_{E_\delta}$ by an affine transformation, that is, 
$$ (R_{\Phi(\theta)} \chi_E) (x) =c (R_\theta \chi_{E_\delta}) (ax+b)$$
for some $a\neq 0$, $c>0$, $b\in \R$.
Therefore $R_{\Phi(\theta)} \chi_E$ is also absolutely continuous, and satisfies $K(\eps, (R_{\Phi(\theta)} \chi_E)') = O(1/\eps^{1/(d-1)})$.
\end{proof}

By combining Theorems~\ref{t:magn}~and~\ref{konvexsuru} we get the following.

\begin{cor}\label{c:convexmagn}
Let $d\ge 2$ and let $E\subset\R^d$ be a bounded convex set with nonempty
interior. Then a set of the form $rE+x$, where $r\ge 1$ and $x\in\R^d$, can be reconstructed using $d+1$ sets.
\end{cor}

Note that by Theorem~\ref{t:interval} the above result fails
for $d=1$.

\section{A general positive result for families with $k$ degrees of freedom}\label{freedom}
In this section we prove that nice geometric objects of $k$ degrees of freedom can be reconstructed using $2k+1$ measurable sets. We also show that this result is sharp.

\begin{notation}
We denote the complete metric space of non-empty compact sets of $\R^d$ 
with the Hausdorff metric by $(\cak(\R^d),d_H)$.
 
In any metric space, let $B(A,\delta)$ denote the open $\delta$-neighborhood 
of the set $A$.

We recall the definition of the upper box dimension (upper Minkowski dimension) and the packing dimension 
in a metric space $X$. The upper box dimension of a bounded set $A\su X$ is
$$
\overline{\dim}_B(A) = \inf\{ s: \limsup_{\eps\to 0} N(A,\eps) \eps^s = 0\},
$$
where $N(A,\eps)$ is the smallest number of $\eps$-balls in $X$ needed
to cover $A$. Recall that in $\R^d$ this is the same as the 
upper Minkowski dimension (see e.g. in \cite{Ma}); that is, 
$$
\overline{\dim}_B(A) = \overline{\dim}_M(A) = 
\inf\{ s: \limsup_{\eps\to 0} \la_d(B(A,\eps))\eps^{s-d} = 0\}.
$$
The packing dimension (or modified upper box dimension in \cite{Fa})
of $A\su X$ is given by 
$$
\dim_P(A) = \inf\left\{\sup_i \overline{\dim}_B(A_i)\ :\ 
A_i \textrm{ is bounded and } A\su \cup_{i=1}^\infty A_i \right\}.
$$
(Alternatively, the packing dimension may be defined in terms of the 
radius based packing measures, see \cite{Cu}.)
\end{notation}

\begin{theorem}\label{special case}
Let $\K$ be a collection of compact subsets in $\R^d$. 
Suppose that ${\dim}_P\, \K \le k$, $k\in\{1,2,\ldots\}$ and for every $K\in\K$,
 $K=\overline{\text{int}\,K}$ and $\overline{\dim}_B \partial K = d-1$.
Then an element of  \ $ \K$ can be reconstructed using $2k+1$ test sets.
\end{theorem}

\begin{remark}
Example~\ref{e:2k} shows that this theorem is sharp in the sense that
$2k+1$ cannot be replaced by $2k$.
\end{remark}

Before proving the theorem we state some corollaries. 
In applications, the condition ${\dim}_P\, \K \le k$
is guaranteed by obtaining $\K$ as a $k$-parameter family of compact subsets
of $\R^d$. More precisely, $\K$ will always be covered by finitely 
many sets of the form $f(G)$, where $G\su\R^k$ is open and 
$f:G\to\cak(\R^d)$ is Lipschitz. This clearly implies ${\dim}_P\, \K \le k$.
Using this observation one can immediately apply Theorem~\ref{special case} 
for any natural collection
of geometric objects with finitely many parameters by counting the 
number of parameters. 
We illustrate this by the following list of applications.
The reader can easily extend this list.

\begin{cor}\label{c:concrete}
\begin{enumerate}
\item \label{5enough}
      An interval in $\R$ can be reconstructed using $5$ test sets.
\item \label{ball}
      A ball in $\R^d$ can be reconstructed using $2d+3$ test sets.
\item An $n$-gon in $\R^2$ can be reconstructed using $4n+1$ test sets.
\item An axis-parallel rectangle in $\R^2$ can be reconstructed using 
$9$ test sets.
\item An ellipsoid in $\R^3$ can be reconstructed using $19$ test sets.
\item A simplex in $\R^d$ can be reconstructed using $2d^2+2d+1$ test sets.
\end{enumerate}
\end{cor}

Instead of Theorem~\ref{special case} we prove the following even more general
statement.

\begin{theorem}\label{randomkonstr}
Let $\K\subset \cak(\R^d)$ be such that ${\dim}_P\, \K <\infty $. Suppose that $K=\overline{\text{int}\,K}$ and that $\overline{\dim}_B \partial K \le b <d$ for every $K\in\K$. 
Then an element of $\K$ can be reconstructed using $r=\left\lfloor \frac{2 \,{\dim}_P \K}{d-b}\right\rfloor + 1$ test sets. 
\end{theorem}

\begin{proof}
We define a random set $A$ and we show that a set $K\in\K$ can be reconstructed
using $r$ independent copies of $A$.

Let $1 > p_1>p_2>\cdots$ be a fast decreasing sequence of reals such that $\sum p_i<\infty$. 
Let $(n_i)$ be an increasing sequence of $2$-powers converging to $\infty$ sufficiently fast.
Let us also assume that $n_{i-1}$ divides $\log_2 n_i$,
 and $\log_2 n_i$ divides $n_i$ for 
each $i$, which conditions
automatically hold if $n_i=2^{2^{l_i}}$ and $l_i$ is a sufficiently fast
increasing sequence of integers.


For each $i$ we take the grids of cubes $\iJ_i=\{(v+[0,1)^d)/n_i\ : v\in\Z^d\}$
and $\iD_i=\{(v+[0,1)^d)/\log_2 n_i\ : v\in\Z^d\}$.
Since $n_{i-1}$ divides $\log_2 n_i$ and $\log_2 n_i$ divides $n_i$,
the partition $\iJ_i$ is finer than $\iD_i$, which is finer than $\iJ_{i-1}$.

Now we define a random set $A_i\subset \R^d$ as the union of certain cubes of
$\iJ_i$ in the following way. 
Independently for each cube $D$ of $\iD_i$ we do the following.
Choose a random integer $m_D$ between $0$ and 
$p_i (n_i/\log_2 n_i)^d$ uniformly.
Then choose randomly $m_D$ cubes of $\iJ_i$ in the cube $D$ (selecting 
each cube with equal probability) and let $H_D$ be their union. 
Finally, let $A_i=\cup_{D\in \iD_i} H_D$.

This way each cube of $\iJ_i$ is contained in $A_i$ with probability approximately $p_i/2$, and points of distance more than $\sqrt{d} / \log_2 n_i$ are independent. (Note the major difference between this random set $A_i$ and the random set which independently chooses each cube of $\iJ_i$ with probability $p_i/2$: The number of $\iJ_i$-cubes of our $A_i$ inside each cube of $\iD_i$ has standard deviation $\approx n_i^d$, while in the other construction it would have standard deviation $\approx \sqrt{n_i^d}$. We ignored $p_i$ here as we will choose $n_i\gg 1/p_i$.)

Since $\sum p_i <\infty$, almost every point of $\R^d$ is contained in only
finitely many sets $A_i$. Hence the following infinite symmetric difference
makes sense (up to measure zero): let $A=A_1 \triangle A_2 \triangle \cdots$.

The key property of this random set is the following.

\bl\label{l:prob}
If $K, K'\in \K$, $K, K'\su [-i,i]^d$ and $K\sm K'$ contains a cube
$D\in \iD_i$ then the probability that
$$
|\lambda(A\cap K) - \lambda(A\cap K')| < \frac{1}{4n_i^d}
$$
is at most $(\log_2 n_i)^{d}/(p_i n_i^{d})$.
\el

\bp
Let $B_i=A_1\triangle \cdots \triangle A_{i}$ ($i=1,2,\ldots$).

First we prove that the probability that
\beq\label{lambdadiff}
|\lambda(B_i \cap K) -
\lambda(B_i \cap K')| 
< 1/(2n_i^d)  
\eeq
is at most $(\log_2 n_i)^d/(p_i n_i^d)$.


Since $D\su K$, we have
\beq\label{eq:difference}
\lambda(B_i \cap K) -
\lambda(B_i \cap K') = 
\lambda(B_i \cap D) +
\lambda(B_i \cap K \cap D^c) -
\lambda(B_i \cap K').
\eeq
Note that the last two terms of the right-hand side depend only on 
$A_1,\ldots,A_{i-1}$ and $A_i\sm D$. Let us fix these random variables.
Then the last two terms are constants, and we know the (conditional) 
distribution of $\lambda(B_i \cap D)$:
this is $m_D/n_i^d$ if $D$ is disjoint from $B_{i-1}$, 
and it is $\la(D)-m_D/n_i^d$ if $D$ is contained in $B_{i-1}$.
Hence the absolute value of the expression of 
\eqref{eq:difference} can be less than $1/(2n_i^d)$ 
only for at most one value of $m_D$. 
Since each value of $m_D$ was chosen with probability at most 
$(\log_2 n_i)^d/(p_i n_i^d)$, this implies that
the conditional probability of \eqref{lambdadiff} is at most  
$(\log_2 n_i)^d/(p_i n_i^d)$.
Since this holds for each fixed choice of $A_1,\ldots,A_{i-1}$ and $A_i\sm D$,  
we get that \eqref{lambdadiff} holds indeed with probability at most  
$(\log_2 n_i)^d/(p_i n_i^d)$.

We can choose $p_{i+1}, p_{i+2}, \ldots$ such that
$$\sum_{j=i+1}^\infty p_j < \frac{1}{100n_i^d (2i)^d}.$$
Then $$\sum_{j=i+1}^\infty \lambda(A_j \cap K) \le  \sum_{j=i+1}^\infty \lambda(A_j \cap [-i,i]^d) < \frac{1}{100n_i^d}$$
since $K\subset [-i,i]^d$ and by construction the density of each $A_j$ is 
at most $p_j$ in each cube of the form $a+[0,1]^d, a\in\Z^d$.
Clearly the same inequality holds for $K'$.

Combining these inequalities with (\ref{lambdadiff})
we get that
$$
|\lambda(A\cap K) - \lambda(A\cap K')| \ge \frac{1}{2n_i^d} - \frac{2}{100n_i^d} \ge  \frac{1}{4n_i^d}
$$
with probability at least $1-(\log_2 n_i)^d/(p_i n_i^d)$.
\ep

Let $s>\dim_P\K$ be such that
$\left\lfloor \frac{2 \,{\dim}_P \K}{d-b}\right\rfloor = \left\lfloor \frac{2s}{d-b}\right\rfloor$.
We may suppose without loss of generality that $\overline{\dim}_B \partial K < b$ for every $K\in\K$ by increasing $b$ such that $\left\lfloor \frac{2s}{d-b}\right\rfloor$ does not increase.
Write $\K$ as $\bigcup_{j=1}^\infty \K'_j$ such that each $\K'_j$ has upper box dimension less than $s$. 

For every $K\in\K$ there exists a positive integer $m_0(K)$ such that for every $m\ge m_0(K)$, 
$$
\lambda(B(\partial K, 1/m)) \le m^{b-d}
$$
since the upper box dimension of $\partial K$ is less than $b$. 
For $K,L\in\K$, using that
$K\triangle L \subset B(\partial K \cup \partial L, d_H(K,L))$,   
this implies that
\beq\label{close}
\lambda(K \triangle L) \le 2 m^{b-d} \quad
\textrm{ if } d_H(K,L)\le 1/m \textrm{ and } m_0(K),m_0(L)\le m.
\eeq

For $i\ge 1$ let
$$
\K_i=\{ K\in \bigcup_{j=1}^i \K'_j \,:\, m_0(K)\le i, \,K\subset [-i,i]^d \}.
$$
Thus $\K_1\subset \K_2 \subset \cdots$, $\bigcup_i \K_i = \K$, and the upper box dimension of each $\K_i$ is less than $s$.

For each $i$ let
$$
\KK_{i}=\{(K,K')\in{\K_i}^2: 
K\sm K' \text{ or } K'\sm K \text{ contains a cube } D\in\iD_i \}.
$$
Then for every integer $N$,
using the assumption that $K=\overline{\text{int}\,K}$ for every $K\in\K$ and thus $\inte K \triangle \inte K' \neq\emptyset$ for every $K\neq K'$, $K,K' \in \K$,
we get that
\beq\label{terfelbontas}
\bigcup_{i=N}^\infty \KK_i = \K^2 \setminus \{(K,K): K\in \K^2\}.
\eeq

Let us fix $i$. 
Since $\K_i$ has upper box dimension less than $s$, for every sufficiently large positive integer $k_i$ (say, for $k_i \ge \kappa_i$) there exist (at most) 
$k_i^s$ sets $\K_i^j\su \K_i$ ($1\le j \le k_i^s$) 
with diameter at most $1/k_i$ that
cover $\K_i$. 

For each pair $(j,j')\in\{1,\ldots,k_i^s\}^2$ pick a pair 
$$
(K_{i,(j,j')},K'_{i,(j,j')}) \in \KK_i \cap (\K_i^j \times \K_i^{j'})
$$
whenever such pair exists.
Then
\beq\label{quantors}
\forall (K,K')\in \KK_i\ 
\exists (j,j')\in\{1,\ldots,k_i^s\}^2\ : \
d_H(K,K_{i,(j,j')}),d_H(K',K'_{i,(j,j')}) \le 1/k_i.
\eeq

Repeating the construction of $A$ independently $r$ times we 
obtain $A_1,\ldots,A_r$. We claim that an element $K\in \K$ can be reconstructed using these sets,
provided we choose the sequences $(n_i)$ and $(p_i)$ appropriately.

For each picked pair $(K_{i,(j,j')},K'_{i,(j,j')})$ we
apply Lemma~\ref{l:prob} to get that the probability that
there exists $1\le t\le r$ such that
\beq\label{probbecsles2}
|\lambda(A^t\cap K_{i,(j,j')}) - \lambda(A^t\cap K'_{i,(j,j')})| \ge \frac{1}{4n_i^d}
\eeq
is at least $1-(\log_2 n_i)^{rd}/(p_i^r n_i^{rd})$.

Since there are at most $k_i^{2s}$ possible pairs $(j, j')$, this implies that with probability at least 
$1-k_i^{2s}(\log_2 n_i)^{rd}/(p_i^r n_i^{rd})$, 
for every picked pair $(K_{i,(j,j')},K'_{i,(j,j')})$ there exists $1\le t\le r$ such that
(\ref{probbecsles2}) holds. 
If
\beq\label{bigki}
k_i\ge \max(i,\kappa_i),
\eeq
then using (\ref{close}) and (\ref{quantors}), this
implies that  with probability at 
least 
$$
1-k_i^{2s}(\log_2 n_i)^{rd}/(p_i^r n_i^{rd}),
$$
for any $(K,K')\in \KK_i$ 
there exists $1\le t\le r$ such that 
\beq\label{probbecsles4}
|\lambda(A^t\cap K) - \lambda(A^t\cap K')| \ge \frac{1}{4n_i^d}-4k_i^{b-d}.
\eeq

Therefore if we choose the sequences $(n_i)$ and $(k_i)$ so that (\ref{bigki}), 
\beq\label{posprob}
\sum_{i=1}^\infty k_i^{2s}(\log_2 n_i)^{rd}/(p_i^r n_i^{rd}) < \infty
\eeq
and
\beq\label{posdiff}
\frac{1}{4n_i^d}-4k_i^{b-d} > 0 \qquad (i=1,2,\ldots)
\eeq
hold then by (\ref{terfelbontas}) and the Borel--Cantelli lemma 
we get that almost surely for
any two distinct $K,K'\in\K$ we have 
$\lambda(A^t\cap K) \neq \lambda(A^t\cap K')$
for at least one $t\in\{1,\ldots,r\}$,
which is exactly what we need to prove.

Choose $k_i$ such that $k_i^{b-d}=n_i^{-d}/64$; that is, $k_i=n_i^{d/(d-b)}/64$. Then (\ref{posdiff}) clearly holds and \eqref{bigki} also holds
if $n_i$ is large enough.
Then, using that $r= \left\lfloor 2s/(d-b) \right \rfloor +1$, we have
$$k_i^{2s}(\log_2 n_i)^{rd}/(p_i^r n_i^{rd}) = 64^{-2s} p_i^{-r} (\log_2 n_i)^{rd} n_i^{d(2s/(d-b) -r)} \le n_i^{-\delta d}$$
for $\delta=(r-2s/(d-b))/2>0$, provided that we choose $n_i$ large enough compared to $1/p_i$.
Since $\delta>0$, this implies that (\ref{posprob}) also holds if $(n_i)$ 
is increasing fast enough. 
This completes the proof of the theorem.
\end{proof}

\section{Open questions}

In this final section we collect some of the numerous remaining open questions.

\begin{question}
How many test sets are needed to reconstruct
an interval in $\R$?
\end{question}

The answer is $3$, $4$ or $5$ by Theorem \ref{t:interval} and Corollary \ref{c:concrete} \eqref{5enough}.

\begin{question}
Let $d\ge 2$. How many test sets are needed to reconstruct a ball in $\R^d$?
For example, does $d+1$ suffice?
\end{question}

We know by Corollary~\ref{c:concrete}~(\ref{ball}) that $2d+3$ test sets are 
enough. 
By Corollary~\ref{c:magnified}~(\ref{largeball}), if we consider only 
balls of radius at least $1$ then the answer is $d+1$ (for $d\ge 2$).
In fact, we also do not know whether the restriction $r\ge 1$ on the
magnification rate is necessary for the other two corollaries 
(\ref{c:magnified}~(\ref{polytope}) and \ref{c:genmagn}) 
of Section~\ref{s:magn}.

\begin{question}
Let $d=1$ or $d=2$. How many test sets are needed to reconstruct
a translate of an arbitrary fixed bounded measurable subset of $\R^d$ of positive measure? 
For example, does $d$ suffice? Does finitely many suffice?
\end{question}

For $d \ge 3$ we know by Corollary \ref{c:posmeasure} that $d$ sets suffice.



\begin{question}
Let $d = 2$ or $d = 3$ and let $E\subset\R^d$ be a bounded set of positive
Lebesgue measure. Can a set of the form $rE+x$ $(r\ge 1, x\in\R^d)$ be
reconstructed using $d+1$ test sets? And finitely many test sets? What if we
drop the condition $r \ge 1$?
\end{question} 

Theorem \ref{t:interval} provides a negative answer for $d = 1$, whereas
Corollary \ref{c:genmagn} shows that the answer is affirmative (with $r \ge
1$) for $d\ge 4$.

\end{document}